\nonstopmode
\documentclass[10pt]{amsart}
\usepackage{graphicx}
\usepackage{latexsym}
\usepackage{fancyhdr}
\usepackage{amsmath, amssymb, amsthm}
\usepackage[all]{xy}
\usepackage{pdflscape}
\usepackage{longtable}
\usepackage{rotating}
\usepackage{verbatim}
\usepackage{hyperref}
\hypersetup{
    linkcolor = blue,
    citecolor = blue,
    urlcolor  = blue,
    colorlinks = true,
}

\usepackage{cleveref}
\usepackage{subfigure}
\usepackage{mathrsfs}
\usepackage{mdwlist}
\usepackage{dsfont}
\usepackage{mathtools}
\usepackage{float}
\usepackage{color}
\usepackage{stmaryrd}

\usepackage{tkz-base}
\usepackage{pgfplots}
\pgfplotsset{compat=newest}

\usepackage{tkz-euclide}
\usepackage{etoolbox}

\definecolor{teal}{rgb}{0.0, 0.5, 0.5}

\newcounter{mnotecount}[section]

\newcommand{\rmnote}[1]{}

\allowdisplaybreaks

\DeclareFontFamily{U}{mathb}{\hyphenchar\font45}
\DeclareFontShape{U}{mathb}{m}{n}{
      <5> <6> <7> <8> <9> <10> gen * mathb
      <10.95> mathb10 <12> <14.4> <17.28> <20.74> <24.88> mathb12
      }{}
\DeclareSymbolFont{mathb}{U}{mathb}{m}{n}
\DeclareFontSubstitution{U}{mathb}{m}{n}

\let\dot\relax
\DeclareMathAccent{\dot}{0}{mathb}{"39}
\let\ddot\relax
\DeclareMathAccent{\ddot}{0}{mathb}{"3A}
\let\dddot\relax
\DeclareMathAccent{\dddot}{0}{mathb}{"3B}
\let\ddddot\relax
\DeclareMathAccent{\ddddot}{0}{mathb}{"3C}

\theoremstyle{plain}
\newtheorem*{theorem*}{Theorem}
\newtheorem{theorem}{Theorem}[section]
\newtheorem*{lemma*}{Lemma}
\newtheorem{lemma}[theorem]{Lemma}
\newtheorem*{assumption*}{Assumption}

\newtheorem*{proposition*}{Proposition}
\newtheorem{proposition}[theorem]{Proposition}
\newtheorem*{corollary*}{Corollary}
\newtheorem{corollary}[theorem]{Corollary}
\newtheorem*{claim*}{Claim}

\newtheorem*{conjecture*}{Conjecture}

\newtheorem*{question*}{Question}
\newtheorem*{result*}{Result}

\theoremstyle{definition}
\newtheorem*{definition*}{Definition}

\newtheorem*{example*}{Example}
\newtheorem{example}[theorem]{Example}

\newtheorem*{algorithm*}{Algorithm}
\newtheorem*{remark*}{Remark}
\newtheorem*{remarks*}{Remarks}
\newtheorem{remark}[theorem]{Remark}

\newtheorem*{convention*}{Convention}

\theoremstyle{plain}
\newtheorem{maintheorem}{Theorem}
\newtheorem{maincorollary}[maintheorem]{Corollary}

\numberwithin{equation}{section}

\def\al{\alpha}
\def\be{\beta}
\def\ga{\gamma}
\def\de{\delta}
\def\ep{\epsilon}

\def\th{\theta}

\def\si{\sigma}

\def\ta{\tau}

\def\vh{\varphi}
\def\ch{\chi}
\def\ps{\psi}
\def\om{\omega}

\def\Ga{\Gamma}

\def\Si{\Sigma}

\def\Ps{\Psi}

\def\N{\mathbb{N}}

\def\R{\mathbb{R}}

\def\cA{\mathcal{A}}

\def\cC{\mathcal{C}}

\def\cH{\mathcal{H}}

\def\cQ{\mathcal{Q}}

\def\sH{\mathscr{H}}

\def\sS{\mathscr{S}}

\def\sU{\mathscr{U}}

\def\p{\partial}
\def\id{\on{id}}

\def\<{\langle}
\def\>{\rangle}
\renewcommand{\o}{\circ}

\def\ol{\overline}
\def\ul{\underline}

\def\wmult{\diamondsuit} 

\let\on=\operatorname

\newcommand{\sr}[1]%
{\ifmmode{}^\dagger\else${}^\dagger$\fi\ifvmode
\vbox to 0pt{\vss
 \hbox to 0pt{\hskip\hsize\hskip1em
 \vbox{\hsize3cm\raggedright\pretolerance10000
 \noindent #1\hfill}\hss}\vss}\else
 \vadjust{\vbox to0pt{\vss%
 \hbox to 0pt{\hskip\hsize\hskip1em%
 \vbox{\hsize3cm\raggedright\pretolerance10000%
 \noindent #1\hfill}\hss}\vss}}\fi%
}

\providecommand{\mapsfrom}{\kern.2em%
\setbox0=\hbox{$\leftarrow$\kern-.10em\rule[0.26mm]{0.1mm}{1.3mm}}\box0%
\kern.3em}

\title[Arc-smooth functions and cuspidality of sets]
{Arc-smooth functions and cuspidality of sets}

\author[A.~Rainer]{Armin Rainer}

\address{Fakult\"at f\"ur Mathematik, Universit\"at Wien,
Oskar-Morgenstern-Platz~1, A-1090 Wien, Austria}
\email{armin.rainer@univie.ac.at}

\begin{document}

\begin{abstract}
  A function $f$ is arc-smooth if the composite $f\o c$ with every smooth curve $c$ in its domain of definition
  is smooth. On open sets in smooth manifolds the arc-smooth functions are precisely the smooth
  functions by a classical theorem of Boman. Recently, we extended this result to certain tame closed
  sets (namely, H\"older sets and simple fat subanalytic sets).
  In this paper we link, in a precise way, the cuspidality of the (boundary of the) set
  to the loss of regularity, i.e., how many derivatives of $f\o c$ are needed in order to determine the derivatives of $f$.
  We also discuss how flatness of $f \o c$ affects flatness of $f$.
  Besides H\"older sets and subanalytic sets we treat sets that are definable in certain polynomially bounded
  o-minimal expansions of the real field.
\end{abstract}

\thanks{Supported by FWF-Project P 32905-N}
\date{\today}

\maketitle


\section{Introduction}

A real valued function $f$ is called \emph{arc-smooth} if $f \o c \in \cC^\infty$
for each $\cC^\infty$-curve $c$ in the domain of definition of $f$.
By a theorem of Boman \cite{Boman67}, the arc-smooth functions defined on open subsets of $\R^d$ (or of smooth manifolds)
are precisely the $\cC^\infty$-functions. In our recent paper \cite{Rainer18}, arc-smooth functions defined on closed sets were studied
and far-reaching extensions of Boman's results were obtained:
\begin{enumerate}
  \item The arc-smooth functions on H\"older sets in $\R^d$ are precisely the restrictions of the $\cC^\infty$-functions on $\R^d$.
  \item The arc-smooth functions on simple fat closed subanalytic sets in $\R^d$ are precisely the restrictions of the $\cC^\infty$-functions on $\R^d$.
\end{enumerate}
We also found analytic and ultradifferentiable analogs.
A precursor for convex fat sets (even in infinite dimensions) is due to Kriegl \cite{Kriegl97}: 
the arc-smooth functions on convex sets $X$ with non-empty interior $X^\o$ are precisely the 
functions that are smooth on $X^\o$ such that the derivatives of all orders on $X^\o$ extend continuously 
to $X$. This is valid for functions between so-called convenient vector spaces, where the interior and continuity are 
understood with respect to the $c^\infty$-topology (which coincides with the trace of the Euclidean topology if $X \subseteq \R^d$; 
cf.\ \Cref{ssec:propertiesHoeldersets}).
Note that closed fat convex sets are $1$-sets (as defined below), in particular, H\"older sets. 

By \emph{H\"older sets} we mean closed sets that satisfy a uniform cusp property; see \Cref{sec:Hoeldersets}.
Recall that a set $X \subseteq \R^d$ is called \emph{fat} if it is contained in the closure of its interior: $X \subseteq \ol{X^\o}$.
It is called \emph{simple} if each $x \in \ol X$ has a basis of neighborhoods $\sU$ such that $U \cap X^\o$ is connected for all $U \in \sU$.
In the terminology of \cite{Rainer18}, the results (1) and (2) mean that H\"older sets and simple fat closed subanalytic sets are
\emph{$\cA^\infty$-admissible} (just as open sets). These closed sets may have cusps (e.g.\ $\{(x,y) \in \R^2 : x\ge 0,\, -x^2 \le y \le x^2\}$)
and horns (e.g.\ $\{(x,y) \in \R^2 : x\ge 0,\, x^2 \le y \le 2 x^2\}$).

The sharpness of the cusps and horns is related to the loss of regularity implicit in (1) and (2).
In fact, to determine the derivatives of order $n$ of an arc-smooth function $f$ at boundary points of $X$,
derivatives of higher order $N$ of the composites $f \o c$ are needed.
In this paper, we will establish an explicit and generally optimal connection between the
cuspidality of the (boundary of the) sets and the discrepancy between $N$ and $n$.
In the process, we shall also see how flatness of $f \o c$ affects flatness of $f$.

In order to formulate our results, we refine the notion of arc-smooth functions to arc-differentiable functions.

\subsection{Arc-differentiable functions}

Let $X$ be a non-empty subset of $\R^d$.
We denote by $\cC^\infty(\R,X)$ the set of all $\cC^\infty$-curves $c : \R \to \R^d$ with $c(\R) \subseteq X$.
For a real valued function $f : X \to \R$ on $X$ we consider
\[
f_* \cC^\infty(\R,X) := \big\{f_* c := f \o c : c \in \cC^\infty(\R,X)\big\}.
\]
For $k \in \N$ and $\be \in (0,1]$
we define
\[
  \cA^{k,\be}(X) := \big\{f : X \to \R : f_* \cC^\infty(\R,X) \subseteq \cC^{k,\be}(\R,\R)\big\}.
\]
The elements of $\cA^{k,\be}(X)$ are called \emph{arc-differentiable functions (of class $\cC^{k,\be}$)} or \emph{arc-$\cC^{k,\be}$ functions} on $X$.
Recall that $\cC^{k,\be}(\R,\R)$ is the space of $k$-times continuously differentiable functions $c : \R \to \R$ such that the $k$-th derivative $c^{(k)}$ is
locally $\be$-H\"older continuous.

For each $\be \in (0,1]$,
\begin{align*}
   \cA^\infty(X) := \big\{f : X \to \R : f_* \cC^\infty(\R,X) \subseteq \cC^{\infty}(\R,\R)\big\} = \bigcap_{k \in \N} \cA^{k,\be}(X).
\end{align*}
The elements of $\cA^\infty(X)$ are the \emph{arc-smooth functions} on $X$.
It is immediate from the definition that arc-smooth mappings can be composed:
if $X_i \subseteq \R^{d_i}$, $i=1,2$, $f \in \cA^\infty(X_2)$, and $\vh = (\vh_1,\ldots,\vh_{d_2}) : X_1 \to X_2$ with $\vh_j \in \cA^\infty(X_1)$, $j=1,\ldots,d_2$,
then $f \o \vh \in \cA^\infty(X_1)$.

Assume that $X \subseteq \R^d$ is a fat closed set so that $X = \ol {X^\o}$.
By definition, $\cC^{k,\be}(X)$ is the set of functions $f : X \to \R$ such that
\begin{enumerate}
  \item $f|_{X^\o} \in \cC^{k}(X^\o)$,
  \item all derivatives $(f|_{X^\o})^{(j)} : X^\o  \to L^j(\R^d,\R)$, $j \le k$, have continuous extensions
  $f^{(j)} : X \to L^j(\R^d,\R)$ to $X$,
  \item $f^{(k)}$ is $\be$-H\"older continuous on compact subsets of $X$.
\end{enumerate}
By $\cC^{k}(X)$ we mean the set of functions $f : X \to \R$ satisfying (1) and (2),
and $\cC^{\infty}(X) := \bigcap_{k \in \N} \cC^{k}(X)$.
These definitions also apply to open sets $X$; in that case (2) is vacuous.
(Note that the definition of $\cC^{\infty}(X)$ differs from the one in \cite{Rainer18}
but in most cases treated in this paper they are equivalent; cf.\ \cite[Lemma 1.10]{Rainer18}.)

Let $Y$ be a subset of $X$. Let $\cA^{k,\be}_Y(X)$ (resp.\ $\cA^{\infty}_Y(X)$) be the set of all
$f \in \cA^{k,\be}(X)$ (resp.\ $f \in \cA^{\infty}(X)$) such that for all $y \in Y$ and all $c \in \cC^\infty(\R,X)$ with $c(0)=y$ we have
\[
  (f \o c)^{(j)}(0) = 0 \quad \text{ for } j \le k \quad\text{ (resp.\ $j \in \N$).}
\]
Similarly, let $\cC^{k,\be}_Y(X)$ (resp.\ $\cC^{\infty}_Y(X)$) be the set of all
$f \in \cC^{k,\be}(X)$ (resp.\ $f \in \cC^{\infty}(X)$) such that
\[
    f^{(j)}|_Y = 0  \quad \text{ for } j \le k \quad\text{ (resp.\ $j \in \N$).}
\]
Then we say that $f$ is \emph{$k$-flat} (resp.\ \emph{$\infty$-flat}) on $Y$.

The $\be$-H\"older condition can be generalized to an $\om$-H\"older condition, 
where $\om$ is a \emph{modulus of continuity},
i.e., an increasing subadditive function $\om : [0,\infty) \to [0,\infty)$ 
such that $\lim_{t \to 0}\om(t) = 0$.
In that case we write $\cA^{k,\om}$, $\cA^{k,\om}_Y$, $\cC^{k,\om}$, and $\cC^{k,\om}_Y$.

\subsection{Main results}

If $X \subseteq \R^d$ is an open set, then $\cA^{k,\be}(X) = \cC^{k,\be}(X)$, by \cite[Theorem 2]{Boman67},
$\cA^{k,\om}(X) = \cC^{k,\om}(X)$, by \cite[Th\'eor\`eme 1]{Faure89}, and consequently $\cA^\infty(X) = \cC^\infty(X)$.
We see that on open sets no loss of regularity occurs.
Moreover, there is no gain (or loss) in considering
smooth plots $p : \R^e \to X \subseteq \R^d$ of arbitrary dimension $e$ in the definition of $\cA^{k,\be}$.
We want to point out that the H\"older condition on the derivative of highest order cannot be omitted without replacement:
$f_* \cC^\infty(\R,X) \subseteq \cC^k(\R,\R)$ does not guarantee $f \in \cC^k(X)$. 
For instance, the map $f : \R^2 \to \R^2$ given by $(x,y) \mapsto \big(\frac{x^3-3x y^2}{x^2+y^2}, \frac{3 x^2 y -y^3}{x^2 + y^2}\big)$ 
(or $(r,\th) \mapsto (r,3\th)$ in polar coordinates)
is not differentiable at the origin, but $f \o c$ is $\cC^1$ for each $C^1$-curve $c : \R \to \R^2$;
cf.\ \cite[Example 3.3]{KM97}.

Let $\al \in (0,1]$. By an \emph{$\al$-set} we mean a closed fat set $X \subseteq \R^d$ such that $X^\o$
has the uniform $\al$-cusp property. If $X$ is compact, then this is equivalent to the fact that $X$
has $\al$-H\"older boundary.
A \emph{H\"older set} is an $\al$-set for some $\al \in (0,1]$.
For precise definitions see \Cref{sec:ucp}.

Note that $\al$ measures the `cuspidality' of an $\al$-set.
We associate two integers with $\al$ which measure the mentioned loss of regularity of arc-smooth functions on $\al$-sets:
\begin{align} \label{eq:defpq}
  p(\al) := \left\lceil\frac{2}{\al} \right\rceil \quad \text{ and } \quad q(\al) := \left\lceil\frac{1}{\al} \right\rceil,
\end{align}
where $\lceil x \rceil$ is the least integer greater or equal $x$.

\begin{maintheorem} \label{main:A}
  Let $X \subseteq \R^d$ be an $\al$-set. Then:
  \begin{enumerate}
    \item For all $\be \in (0,1]$ and $n \in \N$ we have
    \[
    \cA^{np(\al),\be}(X) \subseteq \cC^{n,\frac{\al\be}{2q(\al)}}(X).
    \]
    More generally, if $\om$ is a modulus of continuity, then
    \[
    \cA^{np(\al),\om}(X) \subseteq \cC^{n,\widetilde \om}(X),
    \]
    where $\widetilde \om(t) = \om(t^{\frac{\al}{2q(\al)}})$.
    \item   If $Y$ is any subset of $X$, then $\cA^{np(\al),\om}_Y(X) \subseteq \cC^{n,\widetilde \om}_Y(X)$.
  \end{enumerate}
\end{maintheorem}

Note that $\widetilde \om$, being the composite of two moduli of continuity, is again a modulus of continuity.
The loss of derivatives expressed by $p(\al)$ is essentially best possible, see \Cref{lem:ploss} and \Cref{ex:lossofderivatives}.
In the Lipschitz case, where $\al =1$, $p(\al)=2$, and $q(\al)=1$,
also the degradation of the H\"older index $\be$ to $\frac{\be}2$ is optimal, see \Cref{ex:degradation}.
There are many examples, where the degradation $\frac{\al}{2q(\al)}$ can actually be replaced by $\frac{1}{2q(\al)}$.
We do not know whether in general the worse factor $\frac{\al}{2q(\al)}$ is really necessary or just an artefact of our proof;
see \Cref{rem:square}.

The supplement (2) is particularly interesting for $Y = \p X$.
The reason it is true is that all boundary points $x$ of an $\al$-set $X$ can be reached by $\cC^\infty$-curves in $X$ that
vanish of finite order (at most $p(\al)$) at $x$. See \Cref{ex:irrational} for an interesting set of finite cuspidality
with a boundary point at which all $\cC^\infty$-curves in the set must vanish to infinite order
and hence cannot discriminate points of flatness.

\Cref{main:A} implies the following corollary the first part of which was also obtained in \cite{Rainer18}.

\begin{maincorollary} \label{main:Aa}
  Let $X \subseteq \R^d$ be a H\"older set (and $Y$ any subset of $X$).
  The elements of $\cA^\infty(X)$ (of $\cA^\infty_Y(X)$) are precisely the restrictions to $X$
  of $\cC^\infty$-functions on $\R^d$ (that vanish to infinite order on $Y$).
\end{maincorollary}

In the second part of the paper, we will extend our results to subanalytic sets and, more generally,
to sets that are definable in certain polynomially bounded o-minimal expansions of the real field.
These families of sets also admit horns which are not possible for H\"older sets.
The assumption that the o-minimal structure is polynomially bounded is essential, since
on an infinitely flat cusp $X$ there are functions $f \in \cA^\infty(X)$ that are not of class $\cC^1$,
see \cite[Example 10.4]{Rainer18} and \Cref{ex:lossofderivatives}.

An important requirement of the proof is that the sets under consideration are \emph{uniformly polynomially cuspidal (UPC)}; cf.\ \cite{PawluckiPlesniak86}.
Recall that a closed subset $X \subseteq \R^d$ is UPC if
there exist positive constants $M,m$ and a positive integer $N$ such that
for each $x \in X$ there is a polynomial curve $h_x : \R \to \R^d$ of degree at most $N$
satisfying
\begin{enumerate}
  \item $h_x(0)=x$,
  \item $\on{dist}(h_x(t),\R^d \setminus X) \ge M t^m$ for all $x \in X$ and $t \in [0,1]$.
\end{enumerate}
Note that (2) implies $h_x((0,1]) \subseteq X^\o$.
We may assume that $m$ is a positive integer. We call the reciprocal $\al = \frac{1}{m}$ a \emph{UPC-index} of $X$.
It is again a measure for the `cuspidality' of $X$.

Let $X \subseteq \R^d$ be a fat compact definable set, that means definable with respect to a fixed
polynomially bounded o-minimal expansion of the real field. We shall see that if $X$ admits
\emph{smooth rectilinearization} (defined in \Cref{ass:smoothrect}), then $X$ is UPC and we have the following theorem.

\begin{maintheorem} \label{main:B}
  Let $X \subseteq \R^d$ be a simple fat compact definable set
  admitting smooth rectilinearization
  and let $\al$ be a UPC-index of $X$.
  Let $n \in \N_{\ge 1}$ and $\be \in (0,1]$.
  Then for each $f \in \cA^{np(\al),\be}(X)$ (resp.\ $f \in \cA^{np(\al),\be}_Y(X)$ where $Y$ is any subset of $X$)
  the Fr\'echet derivatives $f^{(p)}$, $p\le n$, are globally bounded on $X^\o$
  and for $p \le n-1$ they extend continuously to $\p X$ (and vanish on $Y$).
  The statement remains true if $\be$ is replaced by an arbitrary modulus of continuity $\om$.
\end{maintheorem}

Being definable in an polynomially bounded o-minimal structure, $X$ is $m$-regular for some $m \in \N_{\ge 1}$
and, consequently, $f \in \cC^{n-1,\frac{1}m}(X)$.

As a consequence of \Cref{main:B} we obtain

\begin{maincorollary} \label{main:C}
  Let $X \subseteq \R^d$ be a simple fat compact definable set
  admitting smooth rectilinearization (and $Y$ any subset of $X$).
  The elements of $\cA^\infty(X)$ (of $\cA^\infty_Y(X)$) are precisely the restrictions to $X$
  of $\cC^\infty$-functions on $\R^d$ (that vanish to infinite order on $Y$).
\end{maincorollary}

In particular, \Cref{main:B} and \Cref{main:C} hold for each simple fat compact $X \subseteq \R^d$ that is
subanalytic or definable in a structure $\R_\cQ$, where $\cQ$ is a suitable quasianalytic class; cf.\ \cite{BM04, RolinSpeisseggerWilkie03}.
It is easy to see (cf.\ \cite[Example 10.5]{Rainer18}) that the assumption that $X$ be simple cannot be
omitted; note that H\"older sets are simple by definition.
We do not know if the assumption of smooth rectifiability is necessary.

As a by-product of our proofs
we obtain a result (\Cref{cor:weakflat}) on \emph{weakly flat} functions on closed UPC sets (not necessarily simple or definable)
in the spirit of \cite{Spallek:1977wp}.

\subsection{Invariance under diffeomorphisms} \label{sec:invariance}

Let $X \subseteq \R^d$ be a closed fat set.
Suppose that $\vh : U \to V$ is a $\cC^\infty$-diffeomorphism between open sets $U,V \subseteq \R^d$ with $X \subseteq U$.
Then $Y := \vh(X)$ is a closed fat set. Fix $k \in \N$ and $\be \in (0,1]$.

Now $f \in \cA^{k,\be}(Y)$ if and only if $f \o \vh \in \cA^{k,\be}(X)$.
Indeed, if $c \in \cC^\infty(\R,Y)$ and $f \o \vh \in \cA^{k,\be}(X)$,
then $\vh^{-1} \o c \in \cC^\infty(\R,X)$ and
$f \o c = f \o \vh \o \vh^{-1} \o c$ is of class $\cC^{k,\be}$.
We also see easily that $f \in \cA^{k,\be}_{\vh(Z)}(Y)$ if and only if $f \o \vh \in \cA^{k,\be}_Z(X)$, where $Z$ is any subset of $X$.

Moreover,
$f \in \cC^{k,\be}(Y)$ if and only if $f \o \vh \in \cC^{k,\be}(X)$
and, more generally,
$f \in \cC^{k,\be}_{\vh(Z)}(Y)$ if and only if $f \o \vh \in \cC^{k,\be}_Z(X)$.
For, suppose that $f \in \cC^{k,\be}(Y)$. Then $f|_{Y^\o} \in \cC^k(Y^\o)$ and $\vh|_{X^\o} : X^\o \to Y^\o$ is $\cC^\infty$.
Thus $f \o \vh|_{X^\o} \in \cC^k(X^\o)$ and its derivatives up to order $k$ can be computed by Fa\`a di Bruno's formula.
In view of this formula we see that $f \o \vh|_{X^\o}$ and all its derivatives up to order $k$ extend continuously to $\p X$,
since this it true for $f$ and $\vh$ and their respective derivatives.
Similarly, one checks that the $k$-th order derivative of $f \o \vh$ satisfies a local $\be$-H\"older condition on $X$.

Clearly, invariance of $\cC^{k,\be}$ on closed fat sets holds even with respect to $\cC^{k,\be}$-diffeomorphisms.
And we have invariance of
\[
\ul \cA^{k,\be}(X) := \big\{f : X \to \R : f_* \cC^{k,\be}(\R,X) \subseteq \cC^{k,\be}(\R,\R)\big\}
\]
under $\cC^{k,\be}$-diffeomorphisms.
Indeed, the composite of $\cC^{k,\be}$-maps is $\cC^{k,\be}$, provided that $k\ge 1$; cf.\ \cite[6.2]{LlaveObaya99} or \cite[Theorem 2.7]{NenningRainer16}.
The inclusions $\ul \cA^{k,\be}(X) \subseteq \cA^{k,\be}(X)$ and $\ul \cA^{\infty}(X) := \bigcap_{k \in \N} \ul \cA^{k,\be}(X) \subseteq \cA^{\infty}(X)$ are evident.
Note that $\sqrt{\cdot} \not\in \ul \cA^{0,1}([0,\infty))$, but $\sqrt{\cdot} \in \cA^{0,1}([0,\infty))$,
by Glaeser's inequality (see \eqref{eq:Glaeser}), but in general it is not clear if the inclusions are strict.
However, as a consequence of \Cref{main:Aa} and \Cref{main:C} we have $\ul \cA^\infty(X) = \cA^\infty(X)$
for simple fat compact definable sets admitting smooth rectilinearization or H\"older sets.

All this is also true if $\be$ is replaced by a general modulus of continuity $\om$.

Since all notions are local, we see that our results continue to hold if the set $X$ in question is replaced by a set
that is locally diffeomorphic to $X$.

\subsection{Related results connecting analytic properties of functions with the geometry of their domain}

The influence of the geometric properties of the boundary of a domain on the analytic aspects of functions on that domain is well-known.
We want to mention some closely related results.
In \cite{Bos:1995wj}, the Sobolev--Gagliardo--Nirenberg inequalities and Markov type inequalities
are shown to be valid on compact subanalytic domains if the inequalities are equipped with a suitable parameter
which measures the cuspidality of the domain.
Markov type inequalities play an important role in approximation theory and differential analysis,
since they are intimately connected to the \emph{Whitney extension property (WEP)}, i.e., the existence
of a continuous linear extension operator of $\cC^\infty$ Whitney jets;
cf.\ \cite{PawluckiPlesniak86,PawluckiPlesniak88,Plesniak:1990aa,Bos:1995wj,Frerick:2007aa}.
UPC sets $X \subseteq \R^d$ satisfy the \emph{Markov inequality} (see \cite{PawluckiPlesniak86}):
there exist positive constants $C,r$ such that for all polynomials $p : \R^d \to \R$ we have
\begin{equation} \label{MI}
    \tag{MI}
   \sup_{x \in X} |\nabla p(x)| \le C (\deg p)^r \sup_{x \in X}|p(x)|.
\end{equation}
A compact set $X \subseteq \R^d$ has WEP if and only if a weaker inequality of Markov type holds (see \cite[Theorem 4.6]{Frerick:2007aa} and \eqref{eq:wMI}).
All sets in our results fulfil WEP,
in particular, they form examples of \emph{Whitney manifold germs} as introduced in \cite{Michor:2020ty};
see also \cite{Roberts:wy} for a related concept.

\subsection{Outline of the paper}

In \Cref{sec:Hoeldersets}, we define H\"older sets and collect some relevant properties.
\Cref{main:A} and \Cref{main:Aa} will be proved in the \Cref{sec:Cn,sec:Hoelder}.
While \Cref{thm:Cn} deals with the continuous extension of derivatives of
arc-differentiable functions to the boundary of a H\"older set,
\Cref{thm:Hoelder0} addresses their H\"older continuity.
In \Cref{sec:optimality}, we explore the optimality and limitations of \Cref{main:A} and show in particular
that the loss of derivatives and degradation of the H\"older index expressed by the integers $p(\al)$ and
$q(\al)$ is generally best possible.
\Cref{sec:definable} is devoted to the study of arc-differentiable functions on definable set,
in particular, to the proofs of \Cref{main:B} and \Cref{main:C}.

\subsection{Notation}

We use $\N = \{0,1,2,\ldots\}$ and $\N_{\ge k} = \{k,k+1,k+2,\ldots\}$.
Let $e_1,e_2,\ldots,e_d$ be the standard unit vectors of $\R^d$.
We endow $\R^d$ with the Euclidean norm $|x| = \big(\sum_{i=1}^d x_i^2\big)^{1/2}$.
The open ball in $\R^d$ with center $x$ and radius $r>0$ is denoted by $B(x,r) = \{y \in \R^d : |x-y|<r\}$.
For a function $f : U \to \R$ defined on an open subset $U$ of $\R^d$
let $f^{(k)}(x)(v_1,\ldots,v_k)$ be the $k$-th order Fr\'echet derivative at $x \in U$ evaluated at the vectors $v_1,\ldots,v_k \in \R^d$.
Then $f^{(k)}(x)$ is an element of the space $L^k(\R^d,\R)$ of $k$-linear mappings $\R^d \times \cdots \times \R^d \to \R$
which we endow with the operator norm $\|\cdot\|_{L^k(\R^d,\R)}$.
We write $d_v^k f(x) := \p_t^k|_{t = 0} f(x +  tv)$ for the $k$-fold directional derivative of $f$ at $x$ in direction $v$.
We will make use of the standard multi-index notation.

\section{H\"older sets} \label{sec:Hoeldersets}

In this section we review the uniform cusp property and H\"older sets.

\subsection{Truncated $\al$-cusps}
Let us consider $\R^d = \R^{d-1} \times \R$ with the Euclidean coordinates $x = (x_1,\ldots,x_d) = (x_{\le d-1},x_d) = (x',x_d)$.
Let $\al \in (0,1]$ and $r,h>0$.
The set
\[
  \Ga_d^{\al}(r,h)
  := \left\{(x',x_d) \in \R^{d-1} \times \R : |x'| < r ,\, h \big(\tfrac{|x'|}{r} \big)^{\al} < x_d < h\right\}
\]
is a \emph{truncated open $\al$-cusp} of radius $r$ and height $h$.

Note that $\Ga_d^{\al}(r,h)$ is the union of the images of all curves $c(t) = (t x', t^{\al}h)$, $t \in (0,1)$, with $|x'| < r$.
We could replace the $(d-1)$-dimensional ball $\{|x'| <r\}$ by an open polyhedron $P \subseteq \R^{d-1}$ containing the origin
and
consider the union $\Pi_d^{\al}(P,h)$ of
the images of all curves $c(t) = (t x', t^{\al}h)$, $t \in (0,1)$, with $x' \in P$.
Then there are radii $r_1 < r_2$ such that
\begin{equation} \label{eq:ballpolyhedron}
\Ga_d^{\al}(r_1,h) \subseteq \Pi_d^{\al}(P,h) \subseteq \Ga_d^{\al}(r_2,h).
\end{equation}

\subsection{Uniform cusp property and $\al$-sets} \label{sec:ucp}

	Let $\al \in (0,1]$.
	We say that an open set $U \subseteq \R^d$ has the \emph{uniform $\al$-cusp property}
	if for every $x \in \p U$ there exist $\ep>0$, a truncated open $\al$-cusp $\Ga = \Ga_d^{\al}(r,h)$, and an
	orthogonal linear map $A \in \on{O}(d)$ such that $y + A\Ga \subseteq U$
	for all $y \in \ol U \cap B(x,\ep)$.

Note that, by \eqref{eq:ballpolyhedron}, we can equivalently replace $\Ga_d^{\al}(r,h)$ by $\Pi_d^{\al}(P,h)$, where $P$ is a polyhedron $P \subseteq \R^{d-1}$ containing the origin.

	By an \emph{$\al$-set} we mean a closed fat set $X\subseteq \R^d$ such that $X^\o$ has the
	uniform $\al$-cusp property.
  Let $\sH^\al(\R^d)$ denote the collection of all $\al$-sets in $\R^d$.
	We say that $X \subseteq \R^d$ is a \emph{H\"older set} if it is an $\al$-set for some $\al \in (0,1]$
  and denote by $\sH(\R^d)$ the collection of all H\"older sets in $\R^d$.
  The elements of $\sH^1(\R^d)$ we also call \emph{Lipschitz sets} in $\R^d$.

Note that $\sH^1(\R^d) \subsetneq \sH^\al(\R^d) \subsetneq \sH^\be(\R^d) \subsetneq \sH(\R^d)$ if $1> \al > \be >0$
(since $\al> \be$ if and only if $\Ga_d^{\al}(r,h) \supsetneq \Ga_d^{\be}(r,h)$).

\begin{remark} \label{rem:Hoelderboundary}
A bounded open set $U \subseteq \R^d$ has the uniform $\al$-cusp property
if and only if $U$ has \emph{$\al$-H\"older boundary};
see \cite[Theorem 6.9, p. 116]{DelfourZolesio11} and \cite[Theorem 1.2.2.2]{Grisvard85}.
That means the following.
At each point $p\in \p U$ there is an orthogonal system of coordinates $(x',x_d)$ and an $\al$-H\"older function $a = a(x')$
such that in a neighborhood of $p$
the boundary of $U$ is given by $\{x_d = a(x')\}$ and the set $U$ is of the form $\{x_d > a(x')\}$.
\end{remark}

	The boundary of an $\al$-set with $\al<1$ can be quite irregular.
  The Hausdorff dimension $\dim_\cH X$ of a compact $X \in \sH^\al(\R^d)$ is not larger than $d-\al$,
  but there are examples $X \in \sH^\al(\R^d)$ with $\dim_\cH X = d-\al$.
  See \cite[Theorem 6.10, p.~116]{DelfourZolesio11}.

\begin{example} \label{ex:alphadomain}
	(1) The closure of a truncated open $\al$-cusp $\Ga^{\al}_d(r,h)$ or of $\Pi_d^{\al}(P,h)$ is an $\al$-set.

  (2) Let $C \subseteq [0,1]$ be the ternary Cantor set and let $f : [0,1] \to \R$ be defined by
	$f(x) := \on{dist}(x,C)^\al$. Then the set $X = \{(x,y) \in \R^2 : -1\le x \le 2,\,  f(x)\le y \le 2 \text{ if } x \in [0,1], \,
	0\le y \le 2 \text{ if } x \not\in [0,1]\}$ is an $\al$-set.

  (3) Convex closed fat sets $X \subseteq \R^d$ are $1$-sets.

	(4) The horn-like set $X = \{(x,y) \in \R^2 : x\ge 0, \,  x^2 \le y \le  2x^{2}\}$ is not a H\"older set,
	but $X$ is the image of the $\frac{1}{2}$-set $\{(x,y) \in \R^2 : x\ge 0, \,  |y| \le \frac{1}{2} x^{2}\}$ under the
	diffeomorphism $(x,y) \mapsto (x,y+\frac{3}{2}x^2)$ of $\R^2$.

	(5) The set $X = \{(x,y) \in \R^2 : x\ge 0, \, x^{3/2} \le y \le 2 x^{3/2}\}$ is not a H\"older set
	and there is no smooth diffeomorphism of $\R^2$ which maps $X$ to a H\"older set. But $X$ is subanalytic.
    The same is true if the exponent $3/2$ is replaced by any positive rational number $r$ that is neither an integer nor 
    the reciprocal of an integer (in which case we can argue as in (4)). 
    Indeed, let $r=p/q$ be such a rational number, where $p$ and $q$ are coprime positive integers.
    We may assume that $p>q$; otherwise we interchange the roles of $x$ and $y$.
    Suppose for contradiction that there is a $\cC^\infty$-diffeomorphism $f : \R^2 \to \R^2$ that 
    maps a H\"older set onto $X$. We may assume without loss of generality that $f(0) = 0$ and that
    $[0,\ep) \times \{0\}$, for some small $\ep>0$,
    is contained in that H\"older set. Then $f(t,0) =: (x(t),y(t))$ is a $\cC^\infty$-curve in $\R^2$ such that 
    $(x(t),y(t)) \in X$ for $t \in [0,\ep)$. 
    Let $n$ be the unique integer such that $n < p/q < n+1$.
    Then $y^{(k)}(0) = 0$ for all $0 \le k \le n$. In fact, $y(0) = 0$, since $f(0) =0$, 
    and if we already know that 
    $y^{(k)}(0) = 0$ for all $0 \le k \le \ell-1<n$, then $y(t) = \frac{1}{\ell !}y^{(\ell)}(s) t^{\ell}$ for some 
    $s \in (0,t)$, so that $(x(t),y(t)) \in X$ for $t \in [0,\ep)$ implies 
    \[
        \frac{x(t)^{\ell}}{t^{\ell}}x(t)^{p/q-\ell} \le \frac{y^{(\ell)}(s)}{\ell!}  \le 2 \frac{x(t)^\ell}{t^\ell}x(t)^{p/q-\ell}, \quad 0 \le t <\ep.
    \]
    Letting $t \to 0$ we may conclude that $y^{(\ell)}(0)= 0$ if $\ell\le n$, because $x(t)/t \to x'(0)$.
    On the other hand, for $\ell = n+1$ we get that $y^{(n+1)}(t)$ (thus also $\p_1^{n+1} f(t,0)$) is unbounded near $t =0$, 
    a contradiction. 

  (6) The flat cusp $X = \{(x,y) \in \R^2 : x\ge 0, \,  |y| \le e^{-1/x^2}\}$ is not a H\"older set.
\end{example}

\subsection{Some properties of H\"older sets} \label{ssec:propertiesHoeldersets}

Let $X \subseteq \R^d$.
The \emph{$c^\infty$-topology} on $X$ is the final topology with respect to all
$\cC^\infty$-curves $c : \R \to \R^d$ satisfying $c(\R) \subseteq X$.
The $c^\infty$-topology on $\R^d$ coincides with the usual topology; cf.\ \cite[Theorem 4.11]{KM97}.

\begin{proposition}[{\cite[Proposition 3.6]{Rainer18}}] \label{topology}
	The $c^\infty$-topology on each H\"older set $X \in  \sH(\R^d)$ coincides with the trace topology from $\R^d$.
\end{proposition}

Recall that a set $X \subseteq \R^d$ is called \emph{$p$-regular}, where $p \in \R_{\ge 1}$,
if each $x \in X$ has a compact neighborhood $K$ in $X$ and there is a constant $D>0$
such that any two points $y_1,y_2 \in K$ can be joined by a rectifiable path $\ga$ contained in $K$ of length
\begin{equation*}
  \ell(\ga) \le D |y_1-y_2|^{1/p}.
\end{equation*}

\begin{proposition}[{\cite[Proposition 3.8]{Rainer18}}] \label{prop:alpharegular}
	Each $X\in \sH^\al (\R^d)$ is $\frac{1}{\al}$-regular.
\end{proposition}

\begin{proposition}[{\cite[Proposition 3.9]{Rainer18}}] \label{alpha-simple}
	Each $X \in \sH(\R^d)$ is simple.
\end{proposition}

\section{Continuous extension of derivatives to the boundary of $\al$-sets}
\label{sec:Cn}

We start with the proof of \Cref{main:A}.
Let $X \subseteq \R^d$ be an $\al$-set and $Y$ any subset of $X$ (possibly the empty set).
In this section, we will show that
\begin{equation*}
   \cA^{np(\al),\om}_Y(X) \subseteq \cC^{n}_Y(X),
\end{equation*}
for each $n \in \N_{\ge 1}$ and each modulus of continuity $\om$.
Then it is easy to complete the proof of \Cref{main:Aa} in \Cref{sec:proofmain:Aa}.

\subsection{Computing derivatives}

We recall a simple formula for the derivatives of composite functions which will be useful below.
We will use the abbreviation $f^{(j)}(x)(v^j) = f^{(j)}(x)(v,\ldots,v)$.

\begin{lemma} \label{lem:Faa}
	Let $1 \le b\le a$ be integers.
	Fix $x \in \R^d$ and $v = (v',v_d) \in \R^d$
	and consider $c(t) = x + (t^{a} v', t^{b}v_d)$, for $t$ in a neighborhood of $0 \in \R$. Let
	$f$ be of class $\cC^{a}$ in a neighborhood of the image of $c$.
	Then:
	\begin{align*}
		\frac{1}{k!} (f\o c)^{(k)}(0) =
			\begin{cases}
				\frac{1}{j!} f^{(j)}(x)((0,v_d)^j) & \text{ if } k= jb < a,
				\\
				f'(x)((v',0)) & \text{ if } k =a \not\in b\N,
				\\
				f'(x)((v',0)) + \frac{1}{j!} f^{(j)}(x)((0,v_d)^j) & \text{ if } k = jb=a.
			\end{cases}
	\end{align*}
	For all other $k<a$ we have $(f\o c)^{(k)}(0)=0$.
\end{lemma}

\begin{proof}
   If $y \in \R^d$ and $\ga(t)= x + t^r y$, then
   \[
    \frac{1}{(rj)!} (f \o \ga)^{(rj)}(0) = \frac{1}{j!} f^{(j)}(x)(y^j)
   \]
   and $(f \o \ga)^{(k)}(0) = 0$ if $k \not\in r\N$.
   From this the lemma follows easily.
\end{proof}

\subsection{The integer $p(\al)$}

Let $\al \in (0,1]$. We define (cf.\ \eqref{eq:defpq})
\begin{align*}
      p(\al) &:= \left\lceil\frac{2}{\al} \right\rceil
      =
      \begin{cases}
          2 & \text{ if } \al = 1,
          \\
          p & \text{ if } \al \in \big[\tfrac{2}{p},\tfrac{2}{p-1}\big), ~ p \in \N_{\ge 3}.
      \end{cases}
\end{align*}

The expedience of $p(\al)$ and its optimality is expressed in next lemma.

\begin{lemma} \label{lem:ploss}
  Let $(v',v_d) \in \Ga^\al_d(r,h)$. Then:
  \begin{enumerate}
    \item $(t^{p(\al)} v',t^2 v_d) \in \Ga^\al_d(r,h)$ whenever $0<|t|\le 1$.
    \item Among all pairs of positive integers $(a,b)$ with $(t^a v',t^b v_d) \in \Ga^\al_d(r,h)$ for arbitrary $(v',v_d) \in \Ga^\al_d(r,h)$ whenever $0<|t|\le 1$
    the pair $(p(\al),2)$ is \emph{minimal} in the sense that
    \[
      p(\al) = \max \{p(\al),2\} \le \max\{a,b\} = a.
    \]
  \end{enumerate}
\end{lemma}

\begin{proof}
  That $(t^a v',t^b v_d) \in \Ga^\al_d(r,h)$ for arbitrary $(v',v_d) \in \Ga^\al_d(r,h)$ whenever $0<|t|\le 1$ is equivalent to
  \[
    |t|^{a \al} \le t^{b} \quad \text{ for all } 0<|t|\le 1.
  \]
  This is in turn equivalent to
  \begin{equation} \label{eq:abal}
    b \in 2 \N_{\ge 1} \quad \text{ and } \quad a \al \ge b.
  \end{equation}
  Taking $b = 2$ we may infer $(1)$.
  Moreover, \eqref{eq:abal} implies $a \ge \frac{b}{\al} \ge \frac{2}{\al}$ and thus $a \ge p(\al)$, that is (2).
\end{proof}

\subsection{Derivatives of arc-differentiable functions}
We first show that arc-differentiable functions have arc-differentiable derivatives and minimize the involved loss of regularity.

\begin{proposition}  \label{Krieglcusp}
  Let $\al \in (0,1]$, $k \ge p(\al)$ an integer, and $\om$ a modulus of continuity.
	Let $X \in  \sH^\al(\R^d)$
	and $f \in \cA^{k,\om}(X)$. Then $f|_{X^\o}$ is of class $\cC^{k,\om}$
	and its derivative $(f|_{X^\o})'$ extends uniquely to a mapping
	$f' : X \to L(\R^d,\R)$ which belongs to $\cA^{k - p(\al), \om}(X, L(\R^d,\R))$, i.e.,
	\begin{equation*}
		(f')_* \cC^\infty(\R,X) \subseteq \cC^{k - p(\al), \om}(\R,L(\R^d,\R)).
	\end{equation*}
  If $Z$ is any subset of $X$ and $f \in \cA^{k,\om}_Z(X)$, then $f' \in \cA^{k - p(\al), \om}_Z(X, L(\R^d,\R))$.
\end{proposition}

\begin{proof}
    That $f|_{X^\o}$ is of class $\cC^{k,\om}$ is well-known; see \cite[Theorem 2]{Boman67} and \cite[Th\'eor\`eme 1]{Faure89}.

  Let us first show that an extension $f' \in \cA^{k - p(\al), \om}(X, L(\R^d,\R))$ of $(f|_{X^\o})'$ is unique:
  Suppose that $f'^{_2} \in \cA^{k - p(\al), \om}(X, L(\R^d,\R))$ is another extension.
  Let $x_0 \in \p X$ and take a $\cC^\infty$-curve $\R \ni s \mapsto x(s)$ in $X$ such that $x(s) \in X^\o$ for $0<|s| \le 1$ and $x(0)= x_0$.
  Then
  \begin{align*}
     f'(x_0) = \lim_{s \to 0} f'(x(s)) = \lim_{s \to 0} f'^{_2}(x(s)) = f'^{_2}(x_0),
  \end{align*}
  since both $f' \o x$ and $f'^{_2} \o x$ are continuous.

  For the existence we observe that, since $X$ is an $\al$-set and the statement is local and invariant under an orthogonal change of coordinates,
  it suffices to show that $(f|_{Y \cap X^\o})'$ extends uniquely to a mapping
	$f' : Y \to L(\R^d,\R)$ which belongs to $\cA^{k - p(\al), \om}(Y, L(\R^d,\R))$,
  where $Y$ is an open subset of $X$ with the following property:
  There is a truncated open $\al$-cusp $\Ga = \Ga_d^{\al}(r,h)$
	such that for all $y \in Y$ we have $y + \Ga \subseteq X^\o$.

	Let $p= p(\al)$, $x \in Y$, and $v = (v',v_d) \in  \Ga$. Set $u = (v',0)$ and $w = (0,v_d)$
  and note that $w \in \Ga$.
	The curves
	\[
		c_{x,v}(t) :=   x+ (t^{p} v',t^{2} v_d) \quad \text{ and } \quad \ell_{x,w} = x + t^2 w
	\]
	both lie in $X^\o$ for $0<|t|<1$ and
	$c_{x,v}(0) = \ell_{x,w}(0) = x$; see \Cref{lem:ploss}.
  Note that $c_{x,v}(t) = \ell_{x,w}(t) + t^p u$.
	Since $f \in \cA^{k,\om}(X)$, the composites $f \o c_{x,v}$ and $f \o \ell_{x,w}$ are of class $\cC^{k,\om}$.

  We will define $f'(x)$ on points $x \in (\p X) \cap Y$ in two steps.

  \textbf{Step 1.}
	Let $v \in \Ga$ be fixed. For $x \in Y$ set
  \begin{align} \label{eq:deffprimev}
     f'(x)(v) :=
     \begin{cases}
        \frac{1}{p!}(f \o c_{x,v})^{(p)}(0) + \frac{1}{2} (f \o \ell_{x,w})^{(2)}(0) & \text{ if } p \not\in 2 \N,
        \\
        \frac{1}{p!}(f \o c_{x,v})^{(p)}(0) - \frac{1}{p!}(f \o \ell_{x,w})^{(p)}(0) + \frac{1}{2} (f \o \ell_{x,w})^{(2)}(0)
        & \text{ if } p \in 2 \N.
     \end{cases}
  \end{align}
	This definition becomes a correct statement if $x \in X^\o$, by \Cref{lem:Faa}. Indeed,
  $\frac{1}{j!} f^{(j)}(x)(w^j) = \frac{1}{(2j)!} (f \o \ell_{x,w})^{(2j)}(0)$ and
  $f'(x)(u) = \frac{1}{p!}(f \o c_{x,v})^{(p)}(0)$, if $p \not\in 2 \N$.
  Otherwise, $p = 2q$ and so
  \begin{align*}
      f'(x)(u) &= \frac{1}{p!}(f \o c_{x,v})^{(p)}(0) - \frac{1}{q!} f^{(q)}(x)(w^{q}) 
      \\
               &= \frac{1}{p!}(f \o c_{x,v})^{(p)}(0) - \frac{1}{p!} (f \o \ell_{x,w})^{(p)}(0).
  \end{align*}

	We claim that
	\begin{equation} \label{smooth}
			\text{$f'(\cdot)(v) :  Y \to \R$ maps $\cC^\infty$-curves to $\cC^{k-p,\om}$-curves.}
	\end{equation}
	Let $\R \ni s \mapsto x(s)$ be a $\cC^\infty$-curve in $Y$.
	Then
	$(s,t) \mapsto c_{x(s),v}(t)$ and $(s,t) \mapsto \ell_{x(s),w}(t)$ are $\cC^\infty$-mappings
  defined on the open strip $\{(s,t) \in \R^2 : |t| < 1\}$ with values in $X$.
  Thus the composites $(s,t) \mapsto f(c_{x(s),v}(t))$ and $(s,t) \mapsto f(\ell_{x(s),w}(t))$ are of class $\cC^{k,\om}$, by \cite[Theorem 2]{Boman67} and \cite[Th\'eor\`eme 1]{Faure89}.
  So, in particular,
	$s \mapsto \p_t^j|_{t=0} f(c_{x(s),v}(t))$ and $(s,t) \mapsto \p_t^j|_{t=0} f(\ell_{x(s),w}(t))$
  are of class $\cC^{k-j,\om}$ for all $j \le k$.
	In view of \eqref{eq:deffprimev} we find that
	$s \mapsto  f'(x(s))(v)$ is of class $\cC^{k-p,\om}$, and the claim is proved.

  Let $x_0 \in (\p X) \cap Y$ and
	let $s \mapsto x(s)$ be any $\cC^\infty$-curve in $Y$ such that $x(s) \in X^\o$ for $0<|s| \le 1$ and $x(0)=x_0$.
	Then $s \mapsto f'(x(s))(v)$ is continuous, by \eqref{smooth}, and hence
	\begin{align*}
		f'(x_0)(v) =  \lim_{s\to 0} f'(x(s))(v).
	\end{align*}

  \textbf{Step 2.}
	Now let $v \in \R^d$ be arbitrary.
	Since $\Ga$ is open, there exist $\ep>0$ and $\xi \in \Ga$ such that $\ep v+\xi \in \Ga$.
	For all $x \in X^\o \cap Y$, we have
	\begin{equation} \label{eq:def2}
    f'(x)(v) =   \frac{f'(x)(\ep v + \xi) - f'(x)(\xi)}{\ep},
  \end{equation}
	and the right-hand side of \eqref{eq:def2}
	extends to points $x \in (\p X) \cap Y$ and satisfies \eqref{smooth}, by the arguments in Step 1.

  Thus for $x_0 \in (\p X) \cap Y$ we define
  \[
  f'(x_0)(v) := \lim_{s \to 0} f'(x(s))(v),
  \]
  where $s \mapsto x(s)$ is a $\cC^\infty$-curve in $Y$
	with $x(0) = x_0$ and $x(s) \in X^\o$ for $0<|s|\le 1$.
  The last paragraph of Step 1 implies that the definition does not depend on the choice of the curve $x$.
  We also see that $f'(x_0)$ is linear as the pointwise limit of $f'(x(s)) \in L(\R^d,\R)$.

	Let us finally show that $f' : Y \to L(\R^d,\R)$ belongs to $\cA^{k-p,\om}(Y,L(\R^d,\R))$.
	Let $x : \R \to Y$ be a $\cC^\infty$-curve and $v \in \R^d$.
	It suffices to check that $s \mapsto f'(x(s))(v)$ is of class $\cC^{k-p,\om}$.
	For $v \in  \Ga$ this follows from \eqref{smooth}.
	For general $v$, $f'(x(s))(v)$ is a linear combination of
	$f'(x(s))(v_1)$ and $f'(x(s))(v_2)$ for $v_i \in  \Ga$
	which locally is independent of $s$.

  To show the supplement assume that $f \in \cA^{k,\om}_Z(Y)$ and let
  $s \mapsto x(s)$ be a $\cC^\infty$-curve in $Y$ with $x(0)=x_0 \in Z$.
  We must prove that
  \begin{equation} \label{eq:Zflat}
    \p_s^j|_{s=0} \big(f'(x(s))(v) \big) = 0 \quad \text{ for } j \le k-p
  \end{equation}
  and all $v \in \R^d$. By construction, it suffices to prove it for $v \in \Ga$.
  To this end consider (as in the paragraph after \eqref{smooth}) the $C^{k,\om}$-functions
  $(s,t) \mapsto f(c_{x(s),v}(t))$ and $(s,t) \mapsto f(\ell_{x(s),w}(t))$
  defined on the open strip $\{(s,t) \in \R^2 : |t| < 1\}$.
  If we compose them with any $\cC^\infty$-curve $r \mapsto (s(r),t(r))$ such that $s(0)=t(0)=0$, we obtain
  functions that vanish to order $k$ at $r=0$, by the assumption on $f$.
  It follows that $(s,t) \mapsto f(c_{x(s),v}(t))$ and $(s,t) \mapsto f(\ell_{x(s),w}(t))$
  vanish to order $k$ at $(s,t)=(0,0)$ (e.g.\ in view of the polarization formula \cite[(7.13.1)]{KM97}).
  Thus, considering \eqref{eq:deffprimev}, we may conclude \eqref{eq:Zflat}.
  The proof is complete.
\end{proof}

\begin{theorem} \label{thm:Cn}
  Let $\al \in (0,1]$, $n \ge 1$ an integer, and $\om$ a modulus of continuity.
  For each $X \in  \sH^\al(\R^d)$ we have $\cA^{np(\al),\om}(X) \subseteq \cC^{n}(X)$.
  If $Y$ is any subset of $X$ then $\cA^{np(\al),\om}_Y(X) \subseteq \cC^{n}_Y(X)$.
\end{theorem}

\begin{proof}
  Let $f \in \cA^{np(\al),\om}(X)$.
  \Cref{Krieglcusp} implies by induction that
  the Fr\'echet derivatives $(f|_{X^\o})^{(m)}$, $m \le n$, have unique extensions
  $f^{(m)} : X \to L^m(\R^d,\R)$ which satisfy
  \begin{equation} \label{eq:derA}
    (f^{(m)})_* \cC^\infty(\R,X) \subseteq \cC^{(n-m)p(\al),\om}(\R,L^m(\R^d,\R)).
  \end{equation}
  Since the $c^\infty$-topology of $X$
  coincides with the trace topology from $\R^d$, by \Cref{topology},
  we may conclude that $f \in \cC^{n}(X)$.
  If $f \in \cA^{np(\al),\om}_Y(X)$, then the supplement of \Cref{Krieglcusp} implies that
  $f^{(m)}|_Y=0$ for $m \le n$.
\end{proof}

\subsection{Proof of \Cref{main:Aa}} \label{sec:proofmain:Aa}

   Let $X \in \sH(\R^d)$ and $f \in \cA^\infty(X)$.
   \Cref{thm:Cn} implies that $f|_{X^\o}$ is smooth and its derivatives of all orders extend continuously to $\p X$.
   By \Cref{prop:alpharegular},
   $f$ defines a Whitney jet of class $\cC^\infty$ on $X$ (cf.\ \cite[Proposition 2.16]{Bierstone80a}
   or the proof of \cite[Lemma 10.1]{Rainer18})
   which has a $\cC^\infty$-extension to $\R^d$, by Whitney's extension theorem \cite{Whitney34a}.
   That, conversely, the restriction to $X$ of any $\cC^\infty$-function on $\R^d$ belongs to $\cA^\infty(X)$ is obvious.

   If $f \in \cA^\infty_Y(X)$, then \Cref{thm:Cn} gives $f \in \cC^{\infty}_Y(X)$ and thus any $\cC^\infty$-extension
   vanishes to infinite order on $Y$.
   Conversely, any $\cC^\infty$-function $f$ on $\R^d$ that vanishes to infinite order on $Y$
   satisfies $(f \o c)^{(j)}(0)=0$ for all $j \in \N$ and all $\cC^\infty$-curves $c$ with $c(0) \in Y$ (by the chain rule).

\section{H\"older continuity of arc-differentiable functions on $\al$-sets}
\label{sec:Hoelder}

In this section we prove

\begin{theorem} \label{thm:Hoelder0}
  Let $\al,\be \in (0,1]$ and $X \in \sH^\al(\R^d)$.
  Then
  \begin{equation}  \label{eq:incltoshow}
      \cA^{0,\be}(X) \subseteq \cC^{0,\frac{\al\be}{2q(\al)}}(X).
  \end{equation}
  More generally, if $\om$ is a modulus of continuity then
  \begin{equation} \label{eq:incltoshowomega}
     \cA^{0,\om}(X) \subseteq \cC^{0,\widetilde \om}(X),
  \end{equation}
  where $\widetilde \om(t) = \om(t^{\frac{\al}{2q(\al)}})$.
\end{theorem}

\Cref{thm:Cn,thm:Hoelder0} imply \Cref{main:A}, since $f^{(n)} \in \cA^{0,\om}(X, L^n(\R^d,\R))$
   in view of \eqref{eq:derA}.
The proof of \Cref{thm:Hoelder0} requires several steps and will fill the rest of the section.
We shall give full details for $\be \in (0,1]$ and comment briefly on the case that $\om$ is a general
modulus of continuity.

   Recall (cf.\ \eqref{eq:defpq}) that
   the integer $q(\al)$ is defined by
   \begin{align*}
         q(\al) := \left\lceil\frac{1}{\al} \right\rceil
         =
         \begin{cases}
             1 & \text{ if } \al = 1,
             \\
             q & \text{ if } \al \in \big[\tfrac{1}{q},\tfrac{1}{q-1}\big), ~ q \in \N_{\ge 2}.
         \end{cases}
   \end{align*}
   Evidently, $p(\al) \le 2 q(\al)$ and  $p(\frac{1}{n}) = 2 q(\frac{1}{n})$ for $n \in \N_{\ge 1}$.

\subsection*{Step 0.}
We will make repeated use of variants of the General Curve Lemma 12.2 in \cite{KM97}.
Let us recall it and fix notation.

Let $s_n \ge 0$ be a sequence of real numbers with $\sum_n s_n < \infty$.
Consider the sequences
\begin{equation} \label{eq:rntn}
   r_n = \sum_{k<n} \left(\frac{2}{k^2} + 2 s_k\right) \quad \text{ and } \quad t_n = \frac{1}{2}(r_n + r_{n+1}).
\end{equation}
as well as the intervals
\begin{equation} \label{eq:intervals}
   I_n := \left[-\frac{1}{n^2} -s_n,\frac{1}{n^2}+s_n\right] \quad \text{ and } \quad J_n := [-s_n,s_n].
\end{equation}
Note that the intervals
$t_n + I_n = [r_n,r_{n+1}]$ have pairwise disjoint interior
and the sequences $r_n$ and $t_n$ have a common finite limit.

\begin{figure}[H]
  \begin{tikzpicture}[scale=1]
		 \tkzInit[xmax=9,ymax=2]

	 \tkzDefPoints{-1/0/a,
	 		 0/0/b,
	 		 2/0/c,
	 		 4/0/d,
       5/0/e,
       6/0/f,
       6.5/0/g,
       7/0/h}

	 \tkzDrawSegment(a,h)

   \tkzDrawPoints(b,d,f)

	 \tkzLabelPoint[below](b){$r_{n-1}$}
   \tkzLabelPoint[below](d){$r_{n}$}
   \tkzLabelPoint[below](f){$r_{n+1}$}

   \tkzLabelPoint[above](c){$t_{n-1} + I_{n-1}$}
   \tkzLabelPoint[above](e){$t_{n}+ I_n$}

   \tkzLabelPoint[left](a){$\cdots$}
   \tkzLabelPoint[right](h){$\cdots$}

	\end{tikzpicture}
\end{figure}

Let $c_n \in \cC^\infty(\R,\R^d)$ be a sequence of $\cC^\infty$-curves
that converges \emph{fast} to $0$, i.e.,
for each $k \in \N$ the sequence $n^k c_n$ is bounded in $\cC^\infty(\R,\R^d)$.
The following fact \cite[12.3]{KM97} (an easy consequence of the Markov inequality on $[-1,1]$) will be useful:
If $c_n$ are polynomials of uniformly bounded degree, then $c_n$ converges fast to $0$ in $\cC^\infty(\R,\R^d)$
provided that the sequence $n \mapsto \sup_{t \in [-1,1]} |c_n(t)|$ converges fast to $0$.

Let $h : \R \to [0,1]$ be a $\cC^\infty$-function with $h|_{\{t \le -1\}} =0$ and $h|_{\{t \ge 0\}} =1$.
Then
\begin{equation} \label{eq:hn}
    h_n(t) := h(n^2(s_n +t))h(n^2(s_n -t))
\end{equation}
has support in $I_n$
and equals $1$ on $J_n$.
It follows that
\begin{equation} \label{eq:c}
  c(t) = \sum_n h_n(t-t_n) c_n(t-t_n)
\end{equation}
defines a curve
$c \in \cC^\infty(\R,\R^d)$ such that $c(t+t_n) = c_n(t)$ for $|t| \le s_n$ and all $n$.

Indeed, at most one summand is non-zero for each $t \in \R$. Since $|h_n^{(j)}(t)| \le n^{2j} H_j$ with $H_j := \max_{t \in \R} |h^{(j)}(t)|$ ,
we have
\begin{align*}
  \MoveEqLeft
   n^2 \sup_{t \in \R} |\p_t^k (h_n(t-t_n) c_n(t-t_n))| =  n^2 \sup_{t \in I_n} |\p_t^k (h_n(t) c_n(t))|
   \\
   &\le n^2 \sum_{j=0}^k \binom{k}{j} n^{2j} H_j \sup_{t \in I_n} |c_n^{(k-j)}(t)|
   \le \left(\sum_{j=0}^k \binom{k}{j} n^{2j+2} H_j\right) \sup_{t \in I_n, \, i\le k} |c_n^{(i)}(t)|.
\end{align*}
Since the right-hand side is bounded in $n$, as $c_n$ converges fast to $0$, $c$ is smooth.

\subsection*{Step 1. Halfspaces and quadrants}

Suppose that $X$ is a \emph{quadrant}
\[
    X=\{x \in \R^d : x_j \ge 0, \, j=i,\ldots,d \},\quad \text{ for some } i = 1,\ldots, d;
\]
this includes the cases of a halfspace ($i=d$) and a `full' quadrant ($i=1$).
In any case $\al = 1$ and $q(1)=1$.
We will show that $\cA^{0,\be}(X) \subseteq \cC^{0,\frac{\be}2}(X)$.

Let $f \in \cA^{0,\be}(X)$.
Suppose for contradiction that $f$ is not locally $\ga$-H\"older near $0 \in X$ with $\ga = \frac{\be}{2}$.
Note that $0$ is the most singular point in $X$; at points in the interior of $X$ we already know that $f$
is even locally $\be$-H\"older and all other boundary points can be treated in a way similar to $0$.
Then there are sequences $a_n$ and $b_n$ in $X$ such that
\begin{align} \label{eq:sequences}
  \begin{split}
    |a_n| \le 4^{-n},~ |b_n| \le 4^{-n},
    \text{ and } \frac{|f(a_n) - f(b_n)|}{|a_n - b_n|^\ga} \ge n 2^{n}  \text{ for all } n.
  \end{split}
\end{align}

We work with quadratic B\'ezier curves.
Let $z_n \in X$ be the point with coordinates $z_{n,j} := \min\{a_{n,j},b_{n,j}\}$, $j = 1,\ldots,d$, and consider
the quadratic B\'ezier curve associated with the triple $(a_n,z_n,b_n)$,
\[
 B_n(t) = z_n + (1-t)^2 (a_n-z_n) + t^2 (b_n-z_n), \quad t \in \R.
\]
Then $B_n$ is a parabola through the points $B_n(0)=a_n$ and $B_n(1) = b_n$
tangent in $a_n$ to the line through $a_n$ and $z_n$ and in $b_n$ to the line through $z_n$ and $b_n$.
If $z_n = a_n$, then the image of $B_n$ is the halfline $a_n + t^2 (b_n -a_n)$; similarly, if $z_n = b_n$.
In any case, $B_n(t)$ is contained in $X$ for all $t \in \R$, by the definition of $z_n$.


Set $s_n^2 := 2^n |a_n - b_n|$ and $c_n(t) := B_n(\frac{t}{s_n})$.
Then the sequence of $\cC^\infty$-curves $c_n$ is contained in $X$ and converges fast to $0$.
Indeed, the $c_n$ are quadratic polynomials and
\begin{equation} \label{eq:Bezierest}
  \max\{|a_n - z_n|,|b_n - z_n|\}  \le |a_n - b_n|  = \frac{s_n^2}{2^n}
\end{equation}
from which the claim follows easily.

We may conclude that the $\cC^\infty$-curve $c$ defined in \eqref{eq:c} lies in $X$,
because each $c_n$ does and multiplication by $h_n$ acts pointwise as a homothety with center $0$ and positive ratio.
Since $c(t+t_n) = c_n(t)$ for $t \in [-s_n,s_n]$,
we have, in view of \eqref{eq:sequences} and as $\be=2\ga$ and $2^{\ga n} \le 2^n$,
\begin{align} \label{eq:contradiction}
   \frac{1}{s_n^\be} |(f \o c)(t_n+s_n) - (f \o c)(t_n)| \ge  \frac{|f(b_n) - f(a_n)|}{2^{n} |a_n - b_n|^{\ga}} \ge n,
\end{align}
contradicting $f \in \cA^{0,\be}(X)$.

Let us briefly indicate how to modify the arguments if $f \in \cA^{0,\om}(X)$, where $\om$ is a modulus of continuity.
For contradiction we may suppose that there are sequences $a_n$ and $b_n$ in $X$ with $|a_n| \le 4^{-n}$, $|b_n| \le 4^{-n}$
and
\[
|f(a_n) - f(b_n)| \ge n 2^{n} \om(|a_n - b_n|^{1/2})
\]
for all $n$. With the same choices as above we find a $\cC^\infty$-curve $c$ in $X$ such that
\begin{align*}
   \frac{1}{\om(s_n)} |(f \o c)(t_n+s_n) - (f \o c)(t_n)| \ge  \frac{|f(b_n) - f(a_n)|}{2^{n} \om(|a_n - b_n|^{1/2})} \ge n,
\end{align*}
since $\om(s_n) = \om(2^{n/2} |a_n - b_n|^{1/2}) \le 2^{n} \om(|a_n - b_n|^{1/2})$ as $\om$ is increasing and subadditive.

\subsection*{Step 2. Simplicial and cubical cusps}

Let $q \in \N_{\ge 1}$.
Consider the unbounded simplicial cusp
\[
  S_q = \{(x_1,\ldots,x_d) \in \R^d : 0 \le x_1 \le x_{2} \le \cdots \le x_{d-1} \le x_d^q \}.
\]
There is the homeomorphism $\vh : S_1 \to S_q$ given by
\[
    \vh(x_1,\ldots,x_d) = (x_1^q,\ldots,x_{d-1}^q,x_d)
\]
and $S_q$ is the union of the curves $(t^q x',t)$, $t \in [0,\infty)$, where $x' \in \{0 \le x_1 \le x_{2} \le \cdots \le x_{d-1} \le 1\}$.
Note that $S_1$ is the image of the quadrant $\{x \in \R^d: x_j \ge 0, \, j =1,\ldots,d\}$
under a linear isomorphism.
If $a,b$ are two different points in $S_q$, we may consider the points $\vh^{-1}(a)$ and $\vh^{-1}(b)$ in $S_1$.
By Step 1,
there is a quadratic curve $B : \R \to S_1$ with $B(0)=\vh^{-1}(a)$ and $B(1)=\vh^{-1}(b)$.
Thus the smooth curve $\vh \o B$ lies entirely in $S_q$ and satisfies $(\vh \o B)(0)=a$ and $(\vh \o B)(1)=b$.

Actually, $S_q$ is not a H\"older set, since the uniform cusp property fails at the tip.
But for technical reasons it is convenient to transit
this intermediate step. Let us show that
\begin{equation} \label{eq:caseSq}
     \cA^{0,\be}(S_q) \subseteq \cC^{0,\frac{\be}{2q}}(S_q).
\end{equation}

Let $f \in \cA^{0,\be}(S_q)$.
Suppose for contradiction that $f$ is not locally $\ga$-H\"older near $0 \in S_q$ with $\ga = \frac{\be}{2q}$;
as in Step 1 it is enough to consider the most singular point.
Then there are sequences $a_n$ and $b_n$ in $S_q$ satisfying \eqref{eq:sequences}.
Let $\vh \o B_n : \R \to S_q$ be the smooth curve satisfying $(\vh \o B_n)(0)=a_n$ and $(\vh \o B_n)(0)=b_n$
that was constructed at the beginning of Step 2.

Set $s_n^{2q} := 2^n |a_n - b_n|$  and $c_n(t) := \vh(B_n(\frac{t}{s_n}))$.
The sequence of curves $c_n$ is clearly contained in $S_q$.
It converges fast to $0$.
Indeed, $c_n$ is a polynomial of degree $2q$, so it suffices to check that $\sup_{t \in [-1,1]} |c_n(t)|$ tends fast to $0$.
Let
$\tilde a_n := \vh^{-1}(a_n)$
and
$\tilde b_n :=\vh^{-1}(b_n)$. Then $\tilde a_{n,j} = a_{n,j}^{1/q}$ and $\tilde b_{n,j} = b_{n,j}^{1/q}$ for $j \le d-1$
as well as $\tilde a_{n,d} = a_{n,d}$ and $\tilde b_{n,d} = b_{n,d}$.
Up to a linear isomorphism we may suppose that $\tilde a_n$ and $\tilde b_n$ lie in the quadrant $\{x \in \R^d: x_j \ge 0, \, j =1,\ldots,d\}$
so that $B_n(t)$ is the quadratic B\'ezier curve from Step 1 with associated control point $\tilde z_n$. We may conclude that
\begin{equation*}
  \max\{|\tilde a_n - \tilde z_n|,|\tilde b_n - \tilde z_n|\}  \le |\tilde a_n - \tilde b_n| \le \sqrt{d}\, |a_n-b_n|^{1/q} = \sqrt{d}\,\frac{s_n^{2}}{2^{n/q}}
\end{equation*}
as in \eqref{eq:Bezierest}, since
\begin{align*}
   |\tilde a_n - \tilde b_n|^2
   &= \sum_{j=1}^{d-1} |a_{n,j}^{1/q}-b_{n,j}^{1/q}|^2 + |a_{n,d}-b_{n,d}|^2
   \\
   &\le   \sum_{j=1}^{d-1} |a_{n,j}-b_{n,j}|^{2/q} + |a_{n,d}-b_{n,d}|^{2/q}
   \le d\, |a_n-b_n|^{2/q}.
\end{align*}
It follows that the maximum of $|B_n(\frac{t}{s_n})|$ on $[-1,1]$ converges fast to $0$
and hence all the more that of $|\vh(B_n(\frac{t}{s_n}))|$.

With the sequences $r_n$ and $t_n$
and the functions $h_n$ from Step 0 we construct a smooth curve $c$ in the following way.
First note that,
if $\ch : \R \to [0,1]$ and $z = (z_1,\ldots, z_d): \R \to S_q$, then the weighted homothety with center $0$ and ratio $\ch$,
\[
   \ch \wmult z := (\ch^q z_1, \ch^q z_2,\ldots,\ch^q z_{d-1},\ch z_{d}),
\]
also takes values in $S_q$. So
\[
 c(t) = \sum_n  (h_n \wmult c_n)(t-t_n)
\]
defines a curve $c : \R \to S_q$ of class $\cC^\infty$, by a calculation similar to the one at the end of Step 0,
since at most one summand is non-zero for each $t \in \R$ and the sequence $c_n$ converges fast to $0$.
In view of $c(t+t_n) = c_n(t)$ for $t \in [-s_n,s_n]$,
we are led to the contradiction \eqref{eq:contradiction}, since $a_n$ and $b_n$ satisfy \eqref{eq:sequences} and as $\be=2 q \ga$.
Thus \eqref{eq:caseSq} is proved.

For a general modulus of continuity $\om$ and $f \in \cA^{0,\om}(S_q)$ we analogously get
\begin{align*}
   \frac{1}{\om(s_n)} |(f \o c)(t_n+s_n) - (f \o c)(t_n)| \ge  \frac{|f(b_n) - f(a_n)|}{2^{n} \om(|a_n - b_n|^{1/(2q)})} \ge n.
\end{align*}

We may symmetrize the result \eqref{eq:caseSq} in the following way.
Set
\[
  \Si_q := \bigcup_{\si \in G} \si S_q,
\]
where $G$ is the group of isometries of $\R^d$ generated by the reflections in the hyperplanes $\{x_j =0\}$, $j=1,\ldots,d-1$, and $\{x_j = x_k\}$, $1 \le j<k \le d-1$.
Then $\Si_q$ is the union of the curves $(t^q x',t)$, $t \in [0,\infty)$, where $x' \in [-1,1]^{d-1}$.
The cubical cusp $\Si_q$ belongs to $\sH^{1/q}(\R^d)$.


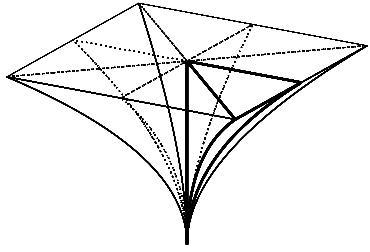
\begin{figure}[H]
\begin{tikzpicture}[scale=0.7]

\begin{axis}
  [   view={120}{25},
      enlargelimits=false,
      xticklabels=\empty,
      yticklabels=\empty,
      zticklabels=\empty,
      axis line style={draw=none},
      tick style={draw=none},
      label style={draw=none}
  ]

\addplot3[variable=t,mesh,color=black,style=very thick,domain=0:1] (t^2,t^2,t);
\addplot3[variable=t,mesh,color=black,style=very thick,domain=0:1] (0,t^2,t);
\addplot3[variable=t,mesh,color=black,style=very thick,domain=0:1] (0,0,t);
\addplot3[variable=t,mesh,color=black,style=very thick,domain=0:1] (0,t,1);
\addplot3[variable=t,mesh,color=black,style=very thick,domain=0:1] (t,t,1);
\addplot3[variable=t,mesh,color=black,style=very thick,domain=0:1] (t,1,1);

\addplot3[variable=t,mesh,color=black,dotted,domain=0:1] (t^2,0,t);
\addplot3[variable=t,mesh,color=black,dotted,domain=-1:1] (t,0,1);
\addplot3[variable=t,mesh,color=black,domain=-1:1] (1,t,1);

\addplot3[variable=t,mesh,color=black,domain=0:1] (t^2,-t^2,t);
\addplot3[variable=t,mesh,color=black,dotted,domain=-1:1] (t,-t,1);
\addplot3[variable=t,mesh,color=black,dotted,domain=0:1] (-t,-t,1);
\addplot3[variable=t,mesh,color=black,dotted,domain=0:1] (0,-t,1);

\addplot3[variable=t,mesh,color=black,domain=-1:1] (t,-1,1);
\addplot3[variable=t,mesh,color=black,domain=-1:1] (-1,t,1);
\addplot3[variable=t,mesh,color=black,domain=0:1] (-t,1,1);

\addplot3[variable=t,mesh,color=black,domain=0:1] (-t^2,t^2,t);
\addplot3[variable=t,mesh,color=black,dotted,domain=0:1] (0,-t^2,t);
\addplot3[variable=t,mesh,color=black,domain=0:1] (-t^2,-t^2,t);
\addplot3[variable=t,mesh,color=black,dotted,domain=0:1] (-t^2,0,t);

\end{axis}

\end{tikzpicture}
\caption{The simplicial cusp $S_q$ and the cubical cusp $\Si_q$ (truncated at $\{x_d = r\}$).}
\end{figure}


We claim that
\begin{equation} \label{eq:caseSiq0}
  \cA^{0,\be}(\Si_q) \subseteq \cC^{0,\frac{\be}{2q}}(\Si_q).
\end{equation}
For,
let $P$ be the union of the hyperplanes that separate the chambers $\si S_q$, $\si \in G$.
Let $f \in \cA^{0,\be}(\Si_q)$, $r>0$, and $y_1, y_2 \in \Si_q \cap B(0,r)$.
Now $\Si_q \cap B(0,r)$ is quasiconvex: There is a rectifiable curve $\pi$ in $\Si_q \cap B(0,r)$ joining $y_1$ and $y_2$
of length $\ell(\pi) \le C \, |y_1-y_2|$, where the constant $C>0$ depends only on $q$, $d$, and $r$.
Let $p_1,p_2,\ldots$ be the points, where $\pi$ intersects $P$.
We may assume that any two consecutive points in the list $p_0:=y_1,p_1,\ldots,p_k,p_{k+1}:=y_2$ belong to the same chamber $\si S_q$
and no chamber contains three or more points in the list, by omitting redundant points. 
Then $k$ is bounded by a constant that depends only on the number of chambers.
Thus, for some constant $D>0$,
\begin{align} \label{eq:symmetric}
   |f(y_1) - f(y_2)|
   &\le \sum_{j = 0}^k |f(p_j) - f(p_{j+1})|
   \\
   &\le \sum_{j = 0}^k D\, |p_j -p_{j+1}|^{\frac{\be}{2q}} \notag
   \\
   &\le D (k+1) \ell(\pi)^{\frac{\be}{2q}}
   \le D C^{\frac{\be}{2q}} (k+1)\,  |y_1 -y_2|^{\frac{\be}{2q}}, \notag
\end{align}
because $f|_{\si S_q} \in \cA^{0,\be}(\si S_q)\subseteq \cC^{0,\frac{\be}{2q}}(\si S_q)$ by \eqref{eq:caseSq}.
Thus \eqref{eq:caseSiq0} is proved.
For a general modulus of continuity $\om$ the same argument applies, since $\om$ is assumed to be increasing and subadditive.

In the next step, we shall need the result for truncated cubical cusps
\[
  \Si_q(r) := \Si_q \cap \{x \in \R^d : x_d \le r\}, \quad r>0.
\]
That is
\begin{equation} \label{eq:caseSiq}
  \cA^{0,\be}(\Si_q(r)) \subseteq \cC^{0,\frac{\be}{2q}}(\Si_q(r)).
\end{equation}
To see this let $f \in \cA^{0,\be}(\Si_q(r))$.
It suffices to show that each point of $\Si_q(r)$ has a neighborhood $U$ in $\Si_q(r)$ such that $f$ is $\frac{\be}{2q}$-H\"older
on $U$.
Observe that near each of its points the set $\Si_q(r)$ is diffeomorphic to an open subset of one of the sets $X$ already treated;
note that the affine hyperplane $\{x_d = r\}$ meets each boundary face of $\Si_q$ transversally.
Using a smooth cut-off function we may assume that there is $\tilde f \in \cA^{0,\be}(X)$ with $\tilde f|_U = f|_U$
so that $f$ is $\frac{\be}{2q}$-H\"older on $U$ (cf.\ \Cref{sec:invariance}).
The reasoning for a general modulus of continuity is the same.

\subsection*{Step 3. The general case}

Let $\al,\be \in (0,1]$ and $X \in \sH^\al(\R^d)$.
Our goal is to show the inclusion \eqref{eq:incltoshow} (and \eqref{eq:incltoshowomega}).

Let $q := q(\al) \in \N_{\ge 1}$, i.e., $\al \in [\frac{1}{q},\frac{1}{q-1})$ if $q \ge 2$ or $\al = 1$ if $q=1$.
In any case, $X \in \sH^{1/q}(\R^d)$, since $\sH^\al(\R^d) \subseteq \sH^{1/q}(\R^d)$.

We will need the truncated cusps $\Si_q(r)$ and
\[
   S_q(r) := S_q \cap \{x \in \R^d : x_d \le r\}
\]
as well as their translates
\[
  \Si_q(y,r) := y + \Si_q(r) \quad \text{ and } \quad S_q(y,r) := y + S_q(r).
\]

First observe that
there is a universal constant $c>0$ such that for sufficiently small $r_1,r_2>0$ we have that
\begin{equation} \label{eq:intersectingcusps}
    |y_1 - y_2| < c \min\{r_1^q,r_2^q\} \text{ implies } S_q(y_1,r_1) \cap S_q(y_2,r_2) \ne \emptyset.
\end{equation}
Indeed, suppose that $r_1 \le r_2$ and initially $y_1 = y_2 =0$ so that $S_q(r_1) = S_q(y_1,r_1) \subseteq S_q(y_2,r_2) = S_q(r_2)$.
Then we take any direction $v \in \mathbb S^{d-1}$ and
move $S_q(r_1)$ in direction $v$ as long as $S_q(tv,r_1) = tv + S_q(r_1)$, $t > 0$, and $S_q(r_2)$ have a common point.
Let $t_v$ be the supremum of such $t$. A lower bound for $\inf_{v \in \mathbb S^{d-1}} t_v$ is
the minimal distance of a vertex of the simplicial cusp $S_q(r_1)$ to its opposite facet
(the only facet that does not contain the vertex as a boundary point).
If $r_1$ is sufficiently small, then the vertex with the minimal
distance $\de$ to its opposite facet is $r_1 e_d$. In the case $q=1$, it is now easy to conclude \eqref{eq:intersectingcusps}.
Suppose that $q>1$.
We clearly have $\de \le r_1^q$.
The point $p$ in the opposite facet that realizes the distance has the form $p_{d}^q e_{d-1} + p_d e_d$
so that
\[
  \de = \sqrt{p_d^{2q} + (r_1-p_d)^2}.
\]
This implies $\de \ge r_1 - p_d$ and thus
$p_d \ge \frac{r_1}{2}$, since otherwise
$\frac{r_1}{2} < \de \le r_1^q$, a contradiction for $r_1 < \frac{1}{2}$.
Consequently, $\de \ge p_{d}^{q} \ge (\frac{r_1}{2})^{q}$ and \eqref{eq:intersectingcusps} follows.
Note that clearly also the symmetric cusps $\Si_q(y_i,r_i)$ have the property \eqref{eq:intersectingcusps}.

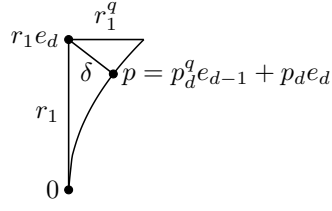
\begin{figure}[H]
\begin{tikzpicture}[domain=0:1, scale = 1]
  \draw[black, line width = 0.20mm]   plot[smooth,domain=0:1] (\x, {2* (\x)^(1/2)});

  \tkzDefPoints{0/0/a,
      0/2/b,
      1/2/c,
      0.594/1.542/p}

  \tkzDrawPoints(a,b,p)

  \tkzDrawSegment(a,b)
  \tkzDrawSegment(b,c)
  \tkzDrawSegment(b,p)

  \tkzLabelPoint[left](a){$0$}
  \tkzLabelPoint[left](b){$r_1 e_d$}
  \tkzLabelPoint[right](p){$p = p_{d}^q e_{d-1} + p_d e_d$}

  \tkzLabelSegment[below,pos=.4](b,p){$\de$}
  \tkzLabelSegment[left,pos=.5](a,b){$r_1$}
  \tkzLabelSegment[above,pos=.5](b,c){$r_1^q$}

\end{tikzpicture}
\end{figure}

Now let $f \in \cA^{0,\be}(X)$.
We will show that each $z \in X$ has a neighborhood on which $f$ is
$\ga$-H\"older with
\[
  \ga = \frac{\al\be}{2q}.
\]
We may assume without loss of generality that $z =0$.
Since $X \in \sH^{1/q}(\R^d)$, we may suppose that (after an affine transformation)
\begin{equation} \label{eq:cone}
  \text{for all $y \in X  \cap B(0,\ep)$ we have $\Si_q(y,1) \subseteq X$}
\end{equation}
provided that $\ep>0$ is sufficiently small.
We know from \eqref{eq:caseSiq} that $f$ is $\ga'$-H\"older on $\Si_q(y,1)$, where
\[
  \ga':=  \frac{\ga}{\al} =  \frac{\be}{2q},
\]
thus in particular on the subsets
\[
  \mathbf \Si_n(y) := \Si_q(y,4^{-n}), \quad n \in \N.
\]
Let $H_n(y)$ be the $\ga'$-H\"older constant of $f$ on $\mathbf \Si_n(y)$, i.e.,
\[
  H_n(y) = \sup_{a \ne b \in \mathbf \Si_n(y)} \frac{|f(a)-f(b)|}{|a-b|^{\ga'}}.
\]
Let $c>0$ be the constant from \eqref{eq:intersectingcusps} and
set
\[
  H_n := \sup \big\{H_n(y) : y \in X  \cap B(0,c\, 4^{-(n+2)q})\big\} \in [0,\infty]
\]
for all large integers $n$ (so that $c\, 4^{-(n+2)q}$ is smaller than the $\ep$ in \eqref{eq:cone}).
We shall distinguish the following two cases:
\begin{enumerate}
  \item[(i)] There exists $n$ such that $H_n<\infty$.
  \item[(ii)] $H_n= \infty$ for all $n$.
\end{enumerate}

\subsubsection*{Case \thetag{i}}

By assumption there is some integer $n$ such that $H_n<\infty$.
Fix this $n$.
Taking $0< \ep_1 \le  c\, 4^{-(n+2)q}$ small enough,
we may assume that in a neighborhood of $B(0,\ep_1)$ the set $X$
is the epigraph $\{x_d \ge \ps(x')\}$ of an $\al$-H\"older function $\ps$, where $x' = (x_1,\ldots,x_{d-1})$
ranges over some convex open set in $\R^{d-1}$; see \Cref{rem:Hoelderboundary}.
Since $\ps$ is uniformly continuous, we may assume that
$\sup_{x',y'} |\ps(x') -\ps(y')| \le 4^{-(n+1)}$
if we shrink $\ep$ if necessary.
Thus, for each $y \in X \cap B(0,\ep_1)$ the cubical cusp $\mathbf \Si_n(y)$
is contained in $X$, by \eqref{eq:cone}, and, having height $4^{-n}$, has non-empty intersection with
\[
  K := \left\{x : x_d > \sup_{x'} \ps(x') + 4^{-(n+1)}\right\}.
\]
We know that $f$ is $\be$-H\"older
on $K$, say with H\"older constant $H$, since $K$ is relatively compact in the interior of $X$.
Let $u=(u',u_d)$ and $v=(v',v_d)$ be any two different points in $X \cap B(0,\ep_1)$.
Then we find $\tilde u \in \mathbf \Si_n(u) \cap K$
and $\tilde v \in \mathbf  \Si_n(v) \cap K$
such that $\tilde u'$ and $\tilde v'$ lie on the line segment $[u',v']$,
$\tilde u_d  = \tilde v_d$,
and
\[
  |u_d - \tilde u_d| \le C\, |u' - \tilde u'|^{\al}
  \quad \text{ and } \quad
    |v_d - \tilde v_d| \le C\, |v' - \tilde v'|^{\al}.
\]
Consequently,
\[
  |u - \tilde u|^2 =   |u' - \tilde u'|^2 + |u_d - \tilde u_d|^2 \le C_1^2\,  |u' - \tilde u'|^{2\al}
\]
and analogously
\[
  |v - \tilde v|^2 \le C_1^2 \, |v' - \tilde v'|^{2\al}.
\]
Since $|\tilde u - \tilde v| = |\tilde u' - \tilde v'|$, we conclude
\begin{align*}
  |f(u) - f(v)| &\le |f(u) - f(\tilde u)| + |f(\tilde u) - f(\tilde v)| + |f(\tilde v) - f(v)|
  \\
  &\le H_n |u - \tilde u|^{\ga'} + H  |\tilde u - \tilde v|^{\be}  + H_n |v - \tilde v|^{\ga'}
  \\
  &\le H_n C_1^{\ga'} |u' - \tilde u'|^{\al \ga'} + H |\tilde u' - \tilde v'|^{\be}  + H_n C_1^{\ga'}  |v' - \tilde v'|^{\al \ga'}
  \\
  &\le C_2\, |u' - v'|^{\ga}
  \\
  &\le C_2\, |u - v|^{\ga}.
\end{align*}
Since $u$ and $v$ were arbitrary, we proved that $f$ is $\ga$-H\"older on $X \cap B(0,\ep_1)$.
So Case~(i) is done.

\subsubsection*{Case \thetag{ii}}

In this case, we prove that $f$ is even $\ga'$-H\"older in a neighborhood of $0 \in X$.
By assumption,
for each sufficiently large $n$
there is a sequence $(y^n_m)_m \subseteq X \cap B(0,c\, 4^{-(n+2)q})$ such that
$\{H_n(y^n_m) : m \in \N\}$ is unbounded.

Let $H^{\si}_n(y^n_m)$ denote the $\ga'$-H\"older constant of $f$ on $\si S_q(y^n_m, 4^{-n})$, where $\si \in G$.
The argument surrounding \eqref{eq:symmetric} shows that
\[
  H_n(y^n_m) \le \on{const} \cdot \sum_{\si \in G} H^{\si}_n(y^n_m).
\]
Since $G$ is finite, we may pass to a subsequence of $m$ and assume that
$H^{\si_n}_n(y^n_m) \to \infty$ as $m \to \infty$ for some $\si_n \in G$.
Passing to a subsequence of $m$ again, we may suppose
that $H^{\si_n}_n(y^n_m) > m 2^{m}$ for all $m \in \N$.
In particular, the diagonal sequence $y_n := y^n_n$ satisfies
$y_n \in X  \cap B(0,c\, 4^{-(n+2)q})$ and
$H^{\si_n}_n(y_n)> n 2^{n}$ for all sufficiently large $n$, say $n \ge N$.
Using the finiteness of $G$ again, we find $\ta \in G$ and
a strictly increasing subsequence $n_k \ge N$ of $n$ such that $\si_{n_k} = \ta$.
Thus $H^{\ta}_{n_k}(y_{n_k})> n_k 2^{n_k}$ for all $k$.
We may assume without loss of generality that $\ta = \id$.
For all $n \ge N$ we set
\[
  \mathbf S_{n} := S_q(y_{n}, 4^{-n}).
\]
So
there exist $a_{n_k} \ne b_{n_k} \in \mathbf S_{n_k}$ such that
\begin{align} \label{eq:contra}
   |f(a_{n_k}) - f(b_{n_k})| \ge n_k 2^{n_k} |a_{n_k} - b_{n_k}|^{\ga'} \quad \text{ for all } k.
\end{align}
Since
\[
  |y_{n} - y_{n+1}| \le |y_{n}| + |y_{n+1}| \le c\, 4^{-(n+2)q} + c\, 4^{-(n+3)q} \le \frac{c}{2}\, 4^{-(n+1)q},
\]
there exists $u_n \in \mathbf S_n \cap \mathbf S_{n+1}$ for all $n \ge N$, by \eqref{eq:intersectingcusps}.
For each $n$ in the complement
of the sequence $(n_k)$ in $\N_{\ge N}$
let $a_n := u_{n-1}$ and $b_n := u_n$; we may assume that $u_{n-1}\ne u_n$.
Then $a_n \ne b_n \in \mathbf S_n$ for all $n\ge N$.

Set $s^{2q}_n = 2^{n} |a_n - b_n|$.
By Step 2, for each $n \ge N$ there exists a $\cC^\infty$-curve $c_n$ such that
\begin{itemize}
     \item $c_n(0)=a_n$, $c_n(s_n)=b_n$,
     \item $c_n$ converges fast to $0$, and
     \item $c_n|_{[0,s_n]}$ is contained in $\mathbf S_n$ and $c_n|_{I_n}$ is contained in $X$, by \eqref{eq:cone},
     where $I_n = [-\frac{1}{n^2} - s_n,\frac{1}{n^2} + s_n]$; cf.\ \eqref{eq:intervals}.
\end{itemize}

Let $r_n$ and $t_n$ be the sequences from \eqref{eq:rntn} and $h_n$ the function defined in \eqref{eq:hn}.
Note that, if $\ch : \R \to [0,1]$ and $z,w : \R \to S_q(y,r)$, then also
$\ch \wmult z + (1-\ch) \wmult w$ takes values in $S_q(y,r)$. Indeed,
up to a translation we may assume that $y = 0$. Then
\begin{align*}
   0 &\le \ch^q z_{1} + (1-\ch)^q w_1,
   \\
   \ch^q z_{j} + (1-\ch)^q  w_{j} &\le \ch^q z_{j+1} + (1-\ch)^q  w_{j+1}, \quad j = 1,\cdots, d-2,
   \\
   \ch^q z_{d-1}  + (1-\ch)^q w_{d-1} &\le \ch^q z_{d}^q  + (1-\ch)^q  w_{d}^q \le (\ch z_{d} + (1-\ch)  w_{d})^q,
   \\
   \ch z_d + (1-\ch)  w_d &\le \ch r + (1-\ch)  r = r.
\end{align*}
We set
\begin{align*}
    A_n(t) &:= (h_n \wmult c_n)(t-t_n),
    \\
    B_n(t) &:= (1-h_n) \wmult \big( u_{n-1} \mathbf{1}_{(r_n,t_n]} + u_n \mathbf{1}_{[t_n,r_{n+1}]}\big)(t-t_n),
    \intertext{where $\mathbf{1}_A$ is the characteristic function of the set $A$, and }
    c(t) &:= \sum_n (A_n(t) + B_n(t)).
\end{align*}
By the above observations and a calculation similar to the one at the end of Step~0,
we see that $c$ is a $\cC^\infty$-curve in $X$;
hereby we use that $u_n$ converges fast to $0$ and $h_n$ has support in $I_n$ and equals $1$ on $[-s_n,s_n]$.
Since $c(t+t_n) = c_n(t)$ for $t \in [-s_n,s_n]$ and $\be=2q\ga'$,
we have, in view of \eqref{eq:contra},
\begin{align*}
   \frac{1}{s_{n_k}^\be} |(f \o c)(t_{n_k}+s_{n_k}) - (f \o c)(t_{n_k})| \ge  \frac{|f(b_{n_k}) - f(a_{n_k})|}{2^{n_k} |b_{n_k} - a_{n_k}|^{\ga'}} \ge {n_k},
\end{align*}
contradicting $f \in \cA^{0,\be}(X)$.
Thus \eqref{eq:incltoshow} is proved.

If $\om$ is a general modulus of continuity, then these arguments also
give a proof of \eqref{eq:incltoshowomega}: it suffices to replace the moduli $t \mapsto t^\ga$ and $t \mapsto t^{\ga'}$
by $\widetilde \om(t) :=  \om(t^{\frac{\al}{2q}})$ and $\widetilde \om'(t) :=  \om(t^{\frac{1}{2q}})$, respectively.
The proof of \Cref{thm:Hoelder0} is complete.

\begin{remark} \label{rem:square}
  We do not know if $\frac{\al}{2q(\al)}$ in the statement of \Cref{thm:Hoelder0} and \Cref{main:A} can be replaced by
  $\frac{1}{2q(\al)}$
  as in the special cases of cusps. The factor $\al$ stems from Case (i) in Step 3, but it is possible that it is just
  an artefact of the proof.
\end{remark}

\section{On the optimality of the results} \label{sec:optimality}

Let us discuss the optimality and limitations of \Cref{main:A}.
The reader's attention is also drawn to
\Cref{ex:irrational} and the examples in \cite[Section 10]{Rainer18}, especially Example 10.4.

\subsection{Loss of derivatives}

By \Cref{main:A}, any function $f \in \cA^{np(\al),\om}(X)$ on an $\al$-set $X$ possesses $n$ continuous Fr\'echet derivatives on $X$.
\Cref{ex:lossofderivatives} shows that the integer $p(\al)$ is optimal in the following sense:
if $p'<p(\al)$ is another integer, then not every function $f \in \cA^{np',\om}(X)$ has $n$ continuous Fr\'echet derivatives on $X$.
Actually, for the set $X$ in \Cref{ex:lossofderivatives} we find $\cA^{np(\al)-3,1}(X) \not\subseteq \cC^{n}(X)$ for suitable $\al$ and $n$.

\begin{example} \label{ex:lossofderivatives}
  Fix $\al \in (0,1)$
  and set
  \[
    X := \big\{(x,y) \in \R^2 : x\ge 0, \, |y| \le x^{1/\al}\big\} \in \sH^\al(\R^2).
  \]
  For $n \in \N$ consider the function $f_n : X \to \R$ defined by
  \[
    f_n(x,y) :=
    \begin{cases}
        \frac{y^{n+1}}{x} & \text{ if } x \ne 0,
        \\
        0 & \text{ if } x=0.
    \end{cases}
  \]
  Then $f_n$ is smooth on $X^\o$ and extends continuously to $\p X$; indeed, $|\frac{y^{n+1}}{x}| \le x^{\frac{n+1}{\al}-1} \to 0$ as $x \to 0$, since $n+1 > \al$.
  Similarly,
  $\p_y^m f_n = \frac{(n+1)!}{(n+1-m)!} f_{n-m}$
  extends continuously to $\p X$ for all $m\le n$, but $\p_y^{n+1} f_n(x,y) = (n+1)!\frac{1}{x}$ does not.
  Moreover, $\p_x^m f_n(x,y) = (-1)^m m!\, \frac{y^{n+1}}{x^{m+1}}$ extends continuously to $\p X$ for all $m \le n$.
  Similarly, one sees that the mixed partial derivatives $\p_x^\ell \p_y^m f_n$ extend continuously to $\p X$ if $\ell+m \le n$.
  That means $f_n \in \cC^{n}(X) \setminus \cC^{n+1}(X)$.
  We also see that $f_n$ is $n$-flat on $\{0\}$.

  Let $c(t) = (x(t),y(t))$ be a $\cC^\infty$-curve in $X$.
  Then $f_n\o c$ is of class $\cC^{\lfloor \be \rfloor}$, where $\be = \frac{2(n+1)}{\al} - 2$ and $\lfloor \be \rfloor$ is the largest integer $\le \be$.
  Consequently, $f_n \in \cA^{\lfloor \be \rfloor-1,1}(X)$.
  This is a consequence of
  the following result \cite[Theorem 7]{JorisPreissmann90}:
  {\it If $\vh,\ps : \R \to \R$ satisfy $\ps \in \cC^\infty$, $\vh \ps \in \cC^\infty$, and $|\vh| \le |\ps|^\ga$ for some positive constant
  $\ga$, then $\vh \in \cC^{\lfloor 2\ga \rfloor}$.}
  Take $\vh = f_n \o c$ and $\ps = x$.

  Now we specify to $\al = \frac{2}{p}$ for some $p \in \N_{\ge 3}$ so that $p(\al)=p$ and $\lfloor \be \rfloor= (n+1)p-2$.
  Thus $f_n \in \cA^{(n+1)p-3,1}(X) \subseteq \cA^{np,1}(X)$. \Cref{main:A} yields $f_n \in \cC^n(X)$, confirming what we checked directly above.
  We also see that $\cA^{(n+1)p-3,1}(X) \not \subseteq \cC^{n+1}(X)$.

  Moreover, if there were a positive integer $p'<p$ such that
  $\cA^{(n+1)p',1}(X) \subseteq \cC^{n+1}(X)$, then
  $f_n \in \cA^{np,1}(X) \subseteq \cA^{(n+1)p',1}(X) \subseteq \cC^{n+1}(X)$ as soon as $n \ge \frac{p'}{p-p'}$, a contradiction.
\end{example}

\begin{remark} \label{rem:unbounded}
  Note that \Cref{ex:lossofderivatives} also shows that
  the partial derivatives of order $n+1$ of a function $f \in \cA^{np(\al),1}(X)$, where $X \in \sH^\al(\R^d)$,
  which exist in $X^\o$ (since $n+1 \le n p(\al)$)
  are in general unbounded at the boundary $\p X$.
\end{remark}

There are closed fat sets $X \subseteq \R^d$ such that each $f \in \cA^\infty(X)$ has a $\cC^\infty$-extension to $\R^d$,
but the loss of derivatives cannot be expressed by an integer
$p$ such that $\cA^{np,1}(X) \subseteq \cC^{n}(X)$ for all $n$:

\begin{example}
  Let
  \[
    K(\al):= \ol {\Ga^{\al}_2(\tfrac{1}{2},1)} = \big\{(x_1,x_2) \in \R^2 : |x_1| \le \tfrac{1}{2},\, (2|x_1|)^\al \le x_2 \le 1\big\}
  \]
  be the truncated closed $\al$-cusp of radius $\frac{1}{2}$ and height $1$ in dimension $2$.
  Then
  \[
    X := \bigcup_{m\in \N_{\ge 1}} \big(m e_1 + K(\tfrac{1}{m})\big) \cup \big\{(x_1,x_2) \in \R^d : x_1 \ge 0, \, 1 \le x_2 \le 2\big\}
  \]
  is an infinite comb with sharper and sharper teeth.
  Each $f \in \cA^\infty(X)$ has a $\cC^\infty$-extension to $\R^2$
  which follows from \Cref{main:Aa}, since $X \cap B(0,n)$ is a H\"older set for each integer $n\ge 1$ and
  the respective extensions can be glued together by a partition of unity.
  On the other hand, we may infer from
  \Cref{ex:lossofderivatives} that there is no positive integer $p$ such that
  $\cA^{np,1}(X) \subseteq \cC^{n}(X)$ for all $n$.
\end{example}

\subsection{Degradation of the H\"older index}

For a $1$-set $X$, \Cref{main:A} yields $\cA^{2n,\be}(X) \subseteq \cC^{n,\frac{\be}{2}}(X)$ for all $n \in \N$ and $\be \in (0,1]$.
That division of the H\"older index by $2$ is optimal is seen by the following example.

\begin{example} \label{ex:degradation}
  Consider the halfspace $X := \{x \in \R^d : x_d \ge 0\}$.
  The function $f : X \to \R$, $x \mapsto (x_d)^{n+\frac{1}{2}}$, belongs to $\cC^{n,\frac{1}2}(X)$, but $f^{(n)}$ is not $\ga$-H\"older
  near $\p X$ for any $\ga>\frac{1}{2}$.
  On the other hand $f \in \cA^{2n,1}(X)$, by Glaeser's inequality \cite[Lemme I]{Glaeser63R}:
  \begin{equation} \label{eq:Glaeser}
     u'(t)^2 \le 2 u(t) \sup_{s \in \R} |u''(s)| \quad \text{ if }\quad u : \R \to [0,\infty).
  \end{equation}
  To see this it suffices check that $u^{n+\frac{1}{2}}$ is of class $\cC^{2n,1}$ if $u \in \cC^\infty(\R,[0,\infty))$; we may assume without loss of
  generality that $u$ has compact support.
  On the set $U:=\{t \in \R : u(t) \ne 0\}$ we may differentiate indefinitely: for $k\ge 1$ we find
  \[
    \p_t^k (u^{n+\frac{1}{2}}) = \sum_{j= 1}^k \sum_{\substack{\al_1 + \cdots + \al_j= k \\  \al_i>0}} C_{j,\al} u^{n-j+\frac{1}{2}} u^{(\al_1)} \cdots u^{(\al_j)},
  \]
  where $C_{j,\al}$ are numerical constants.
  We claim that for each $k \le 2n$ all summands on the right-hand side extend continuously by $0$ to the complement of $U$.
  This is clear for $k\le n$, because then the exponent $n-j+\frac{1}{2}$ is positive.
  If $n < j \le k \le 2n$, then $\al_1 + \cdots + \al_j = k$ implies that at least $2(j - n)$ among the $\al_i$ must equal $1$ so that
  \[
    u^{n-j+\frac{1}{2}} u^{(\al_1)} \cdots u^{(\al_j)}  = \frac{(u')^{2(j-n)}}{u^{j-n-\frac{1}{2}}} \cdot P 
    =u' \cdot \left( \frac{u'}{u^{\frac{1}{2}}} \right)^{2(j-n)-1} \cdot P,
  \]
  where $P$ is the product of the remaining factors. By \eqref{eq:Glaeser}, $u'/u^{\frac{1}{2}}$ is bounded and $u'$ vanishes on the complement of $U$.
  Thus $u^{n+\frac{1}{2}}$ is of class $\cC^{2n}$.
  For $k = 2n+1$ a similar argument shows that at least $2(j - n)-1$ among the $\al_i$ must equal $1$ so that all summands
  are globally bounded on $U$.
  That means that $v:= \p_t^{2k} (u^{n+\frac{1}{2}})$ is Lipschitz on each connected component of $U$ with uniform Lipschitz constant $L$.
  Since $v$ is zero on the complement of $U$, it follows that $v$ is Lipschitz on $\R$.
  Indeed, if $t_1 < t_2$ are not in the same component, take $s_1 \le s_2$ in $[t_1,t_2]$ such that
  $s_i$ is an endpoint of the component of $t_i$ if $t_i \in U$ and $s_i = t_i$ otherwise.
  Thus, $v(s_i)=0$ for $i =1,2$ and
  \[
    |v(t_2) - v(t_1)| \le |v(t_2) - v(s_2)| + |v(s_1) - v(t_1)| \le L(t_2-s_2) + L(s_1- t_1) \le L(t_2-t_1).
  \]
\end{example}

\subsection{Locally finite unions of $\al$-sets}

The conclusion of \Cref{main:A} can be (partially) extended to locally finite unions of $\al$-sets if
the overlaps of the pieces are not too ``thin''. Intersections do in general not preserve the conclusion as
is shown by the infinitely flat cusp \cite[Example 10.4]{Rainer18}.

\begin{theorem} \label{thm:union}
  Let $X \subseteq \R^d$ be a locally finite union of $\al$-sets $X_j$ such that:
  \begin{enumerate}
    \item[($\star$)] If $x \in \p X$ and $x \in X_i \cap X_j$, then there exists a non-empty $\al$-set $Y$ such that
    $x \in Y \subseteq X_i \cap X_j$.
  \end{enumerate}
  Then $\cA^{n p(\al), \be}(X) \subseteq \cC^{n}(X)$ for all $\be \in (0,1]$ and $n \in \N$.
  If, additionally,
  \begin{enumerate}
    \item[($\star\star$)] $X$ is $m$-regular,
  \end{enumerate}
  then $\cA^{n p(\al), \be}(X) \subseteq \cC^{n, \frac{\al \be}{2q(\al) m}}(X)$.
  Note that $\be$ may be replaced by a general modulus of continuity $\om$.
\end{theorem}

\begin{proof}
  Let $f \in \cA^{n p(\al), \be}(X)$.
  Then $f|_{X_j} \in \cA^{n p(\al), \be}(X_j)$ for each $j$.
  Since $X_j$ is an $\al$-set, we have $f|_{X_j} \in \cC^n(X_j)$ for each $j$, by \Cref{main:A}.
  It remains to check that the derivatives up to order $n$ of $f|_{X_j}$ and $f|_{X_i}$
  for $j \ne i$ coincide at points $x \in \p X$
  which belong to $X_j \cap X_i$.
  But that follows from condition ($\star$), since the derivatives are uniquely determined by the restriction $f|_Y$.
  Note that ($\star$) implies that $X$ is simple.

  Fix $a \in X$. By ($\star\star$), there is a compact neighborhood $K \subseteq X$ of $a$
  such that any two points $x,y$ in $K$ can be joined by a rectifiable path $c$ in $K$ with
  \[
    \ell(c) \le C |x-y|^{1/m}.
  \]
  There is only a finite number of $X_j$ with $K \cap X_j \ne \emptyset$.
  Let $H$ be the maximum of the $\ga:= \frac{\al \be}{2q(\al)}$-H\"older constants of $f^{(n)}|_{X_j}$; see \Cref{main:A}.
  Then we find a sequence of points $x_0 :=x,x_1,\ldots,x_k:=y$ on $c$
  such that any two consecutive points $x_i,x_{i+1}$ belong to the same $X_j$
  and these two are the only points in the sequence that are contained in $X_j$.
  Thus
  \begin{align*}
     \|f^{(n)}(x) - f^{(n)}(y)\|_{L^n(\R^d,\R)} &\le \sum_{i=0}^{k-1} \|f^{(n)}(x_i) - f^{(n)}(x_{i+1})\|_{L^n(\R^d,\R)}
     \\
     &\le H \sum_{i=0}^{k-1} |x_i - x_{i+1}|^{\ga} \le H k\, \ell(c)^{\ga} \le C^\ga H k\, |x-y|^{\frac{\ga}{m}}.
  \end{align*}
  The reasoning for an arbitrary modulus of continuity is analogous.
\end{proof}

Without an additional condition such as ($\star\star$) we cannot expect that the derivatives of any order satisfy a H\"older condition
as is seen by the following example.

\begin{example}
 \label{ex:complementflatcusp}
 Cf.\ \cite[Example 10.9]{Rainer18} and \cite[Example 2.18]{Bierstone80a}.
    Let $X$ be the complement in $\R^2$ of the flat cusp $\{(x,y) \in \R^2 : x>0,\, |y| < e^{-1/x^2}\}$.
    It was observed in \cite[Example 10.9]{Rainer18} that $\cA^\infty(X) = \cC^\infty(X)$.
    Indeed, we have $\cA^{2n, \be}(X) \subseteq \cC^{n}(X)$ for all $n \in \N$ and $\be \in (0,1]$, by \Cref{thm:union},
    since $X$ is the union of the two $1$-sets
    \[
  		X_\pm := \big\{(x,y) \in \R^2 : x>0,\, \pm y \ge e^{-1/x^2} \big\} \cup \big\{(x,y) \in \R^2 : x\le 0\big\},
  	\]
    and $X_+ \cap X_- = \{(x,y) \in \R^2 : x\le 0 \}$ is also a $1$-set.

    Consider the function $f : X \to \R$ defined by
    \[
      f(x,y) :=
      \begin{cases}
          e^{-1/x} & \text{ if } x>0,\, y \ge e^{-1/x^2},
          \\
          e^{-2/x} & \text{ if } x>0,\, y \le - e^{-1/x^2},
          \\
          0 & \text{ if } x\le 0.
      \end{cases}
    \]
    Then $f \in \cA^\infty(X) = \cC^\infty(X)$.
    But
    \begin{align*}
          \frac{|f(x,e^{-1/x^2}) - f(x,-e^{-1/x^2})|}{|(x,e^{-1/x^2}) - (x,-e^{-1/x^2})|^\ga}
          =\frac{e^{-1/x}(1 - e^{-1/x})}{2^\ga e^{-\ga/x^2}} \to \infty \quad \text{ as } x\searrow 0,
    \end{align*}
    that is,
    $f$ is not $\ga$-H\"older near $0 \in X$ for any $\ga \in (0,1]$.
    Since $\p_x^{n} f(x,e^{-1/x^2}) = p_1(\frac{1}{x}) e^{-1/x}$ and $\p_x^{n} f(x,-e^{-1/x^2}) = p_2(\frac{1}{x}) e^{-2/x}$ for polynomials $p_1,p_2 \in \R[x]$,
    we may likewise conclude that
    $\p_x^{n} f$ is not $\ga$-H\"older near $0 \in X$ for any $\ga \in (0,1]$.
    In particular, $f$ is not the restriction to $X$ of a $\cC^\infty$-function on $\R^2$.
\end{example}

\subsection{Other building blocks}
One might be tempted to consider smooth curves defined on $\R_+ := [0,\infty)$ as basic building blocks.
Given $X \subseteq \R^d$
let $\cC^\infty(\R_+,X)$ be the set of all curves $c : \R_+ \to \R^d$ with $c(\R_+) \subseteq X$ that have
a $\cC^\infty$-extension $\tilde c : \R \to \R^d$ (we do not require $\tilde c(\R) \subseteq X$)
and set
\[
  \cA^{k,\be}_+(X) := \big\{f : X \to \R : f_* \cC^\infty(\R_+,X) \subseteq \cC^{k,\be}(\R_+,\R) \big\}.
\]
By definition, $\cA^{k,\be}_+(\R_+) = \cC^{k,\be}(\R_+)$ and thus $\cA^{k,\be}_+(\R_+) \ne \cA^{k,\be}(\R_+)$, by \Cref{ex:degradation}.
The notions are however not too far apart:

\begin{proposition}
  Let $k \in \N$, $\be \in (0,1]$,
  and $X \subseteq \R^d$ arbitrary.
  Then:
  \begin{enumerate}
    \item $\cA^{k,\be}_+(X) \subseteq \cA^{k,\be}(X)$ and $\cA^{2k,\be}(X) \subseteq \cA^{k,\frac{\be}2}_+(X)$.
    \item $\cA^{\infty}_+(X) := \bigcap_{k \in \N} \cA^{k,\be}_+(X) = \cA^{\infty}(X)$.
    \item If $X \subseteq \R^d$ is open, then $\cA^{k,\be}_+(X) = \cA^{k,\be}(X)$.
  \end{enumerate}
  Analogous versions hold if $\be$
  is replaced by a general modulus of continuity.
\end{proposition}

\begin{proof}
  (1)
  Suppose that $f \in \cA^{k,\be}_+(X)$ and let us show $f \in \cA^{k,\be}(X)$.
  For any $\cC^\infty$-curve $c$ in $X$ the curves
  $c_+:= c|_{\R_+}$ and $c_{-}(t) := c(-t)$ for $t\ge 0$ belong to $\cC^\infty(\R_+,X)$. Then
  \[
  	(f\o c)(t) =
  	\begin{cases}
  		(f \o c_+)(t) & \text{ if } t \ge 0,
  		\\
  		(f \o c_-)(-t) & \text{ if } t \le 0,
  	\end{cases}
  \]
  is of class $\cC^{k,\be}$ off $0$.
  To see that $f\o c$ is of class $\cC^{k,\be}$ at $0$ observe that
  $c_1(t) := c(t-1)$ for $t \ge 0$ belongs to $\cC^\infty(\R_+,X)$ and
  \[
  	(f\o c)(t) = (f\o c_1)(t+1) \quad \text{ for } t \ge -1.
  \]

  To prove the second inequality in (1) we take an arbitrary $f \in \cA^{2k,\be}(X)$ and $c \in \cC^\infty(\R_+,X)$ and show that $f \o c \in \cC^{k,\frac{\be}2}(\R_+,\R)$.
  For any $\cC^\infty$-curve $\ga : \R \to \R_+$ the composite
  \[
  	(f \o c) \o \ga = f \o (c \o \ga)
  \]
  is of class $\cC^{2k,\be}$, i.e., $f \o c \in \cA^{2k,\be}(\R_+)$.
  Thus $f\o c \in \cC^{k,\frac{\be}2}(\R_+)$,
  by \Cref{main:A}.

  (2) is a direct consequence of (1).

  (3) In view of (1), it suffices to show $\cA^{k,\be}(X) \subseteq \cA^{k,\be}_+(X)$.
  If $X$ is open, then $\cA^{k,\be}(X) = \cC^{k,\be}(X)$ so that the desired inclusion
  follows from the fact that the composite of a $\cC^{k,\be}$-function with a $\cC^\infty$-curve
  is a $\cC^{k,\be}$-function.
\end{proof}

\section{Arc-differentiable functions on definable sets} \label{sec:definable}

In this section, we will study arc-differentiable functions on sets that are definable in
polynomially bounded o-minimal expansions of the real field.
We shall prove \Cref{main:B} and \Cref{main:C}.

\subsection{Polynomially bounded o-minimal expansions of the real field}

We recall the definition of an \emph{o-minimal structure} over the real (ordered) field and some background;
cf.\ \cite{vandenDriesMiller96} and \cite{vandenDries98}.

A \emph{structure} $\sS = (\sS_d)_{d\ge 1}$ over the real (ordered) field $(\R,+,\cdot)$ is a sequence, where
each $\sS_d$ is a collection of subsets of $\R^d$, such that for all $d,d'\ge 1$:
\begin{itemize}
  \item $\sS_d$ is a boolean algebra with respect to the usual set-theoretic operations.
  \item $\sS_d$ contains all semialgebraic subsets of $\R^d$.
  \item If $X \in \sS_{d}$ and $X' \in \sS_{d'}$, then $X \times X' \in \sS_{d+d'}$.
  \item If $d \ge d'$ and $X \in \sS_{d}$, then $\pi(X) \in \sS_{d'}$, where $\pi : \R^{d}\to \R^{d'}$
  is the projection on the first $d'$ coordinates.
\end{itemize}
A subset $X \subseteq \R^d$ is said to be \emph{definable} in the structure $\sS$ if $X \in \sS_d$.
A map $f : X \to \R^{d'}$ is called \emph{definable} in $\sS$ if its graph is definable.
A structure $\sS$ is called \emph{o-minimal} if
\begin{itemize}
  \item the boundary of every set in $\sS_1$ is finite.
\end{itemize}

A structure $\sS$ is called \emph{polynomially bounded} if for every function $f : \R \to \R$ that is definable in $\sS$
there exists $N \in \N$ such that $f(t) = O(t^N)$ as $t \to \infty$.
An o-minimal structure $\sS$ either is polynomially bounded or the exponential function $\exp : \R \to \R$ is definable in $\sS$
(see \cite{Miller:1994ue}).

Here are a few examples of o-minimal structures relevant for this paper:
\begin{enumerate}
  \item The collection of all semialgebraic sets in $\R^d$ for $d \ge 1$ is a polynomially bounded o-minimal structure.
  \item The family of globally subanalytic sets in $\R^d$ for $d \ge 1$ is a polynomially bounded o-minimal structure.
  It is the smallest structure over $(\R,+,\cdot)$ containing all restricted analytic functions
  $f : \R^d \to \R$, i.e., $f|_{[-1,1]^d}$ is analytic and $f=0$ outside $[-1,1]^d$.
  It is denoted by $\R_{\text{an}} := (\R,+,\cdot,(f)_{f \text{ restricted analytic}})$.
  \item The expansion $\R^\R_{\text{an}} := (\R_{\text{an}},(x^r)_{r \in \R})$ of $\R_{\text{an}}$
  by all real powers $x^r : \R \to \R$, $t \mapsto t^r$ if $t >0$ and $t\mapsto 0$ if $t\le 0$,
  is a polynomially bounded o-minimal structure.
  \item The expansion $\R_{\text{an},\exp} := (\R_{\text{an}},\exp)$ of $\R_{\text{an}}$ by
  the unrestricted exponential function $\exp : \R \to \R$ is an o-minimal structure which is not polynomially bounded.
\end{enumerate}

In this paper, we will be concerned only with polynomially bounded o-minimal structures.
In fact, if the exponential function is definable, then infinitely flat cusps are definable and on such arc-differentiable functions need not
be of class $C^1$; see \cite[Example 10.4]{Rainer18}.

\emph{From now on we suppose that $\sS$ is an arbitrary polynomially bounded o-minimal structure over the real field.
If we say that a set or a map is definable, we mean definable in $\sS$ (unless stated otherwise).
}

\subsection{{\L}ojasiewicz inequality}
Cf.\ \cite[4.14]{vandenDriesMiller96}.
Let $f,g : X \to \R$ be continuous definable functions on a compact definable set $X$ such that $f^{-1}(0) \subseteq g^{-1}(0)$.
Then there exist constants $N,C>0$ such that
\begin{equation} \label{eq:Lojasiewicz}
   |g(x)|^N \le C |f(x)|, \quad  x \in X.
\end{equation}

\subsection{Whitney regularity} \label{sec:Wregular}
Cf.\ \cite[4.15]{vandenDriesMiller96}.
Let $X \subseteq \R^d$ be a compact connected definable set.
There exist $r,C>0$ and a definable map $\ga : X^2 \times [0,1] \to X$ such that, for all $x,y \in X$,
$[0,1] \ni t \mapsto \ga(x,y,t) \in X$ is a rectifiable path from $x$ to $y$ of length $\le C |x-y|^r$.

\subsection{Quasiconvex decomposition}
Any definable set $X \subseteq \R^d$ can be decomposed into a finite disjoint union $X = \bigcup_j X_j$ of
$M$-quasiconvex sets $X_j$,
where $M$ depends only on the dimension $d$; see \cite[Theorem 1.2]{KurdykaParusinski06}.
Here a set $Y \subseteq \R^d$ is called \emph{$M$-quasiconvex} if any two points $y_1,y_2 \in Y$ can be joined
in $Y$ by a piecewise smooth path of length at most $M |y_1 - y_2|$. (Note that for this result the underlying o-minimal structure
need not necessarily be polynomially bounded.)

\begin{lemma} \label{lem:quasiconvex}
	Let $X \subseteq  \R^d$ be a fat closed definable set. Let $x \in \p X$ and suppose there is a
	basis of neighborhoods $\sU$ of $x$ such that $U \cap X^\o$ is connected for all $U \in \sU$.
	For all $U \in \sU$ and
  any two points $y,z \in U \cap X^\o$, there exists a
	rectifiable path $\ga$ in $X^\o$ from $y$ to $z$ of length
	\[
		\ell(\ga) \le C  \on{diam}(U).
	\]
\end{lemma}

\begin{proof}
   We only sketch the argument;
   details can be found in the proof of \cite[Lemma 5.9]{Rainer18}.
   There is a finite disjoint decomposition of $X^\o$ into
   $M$-quasiconvex sets $A_j$. Fix $U \in \sU$ and $y,z \in U \cap X^\o$.
   There is a path $\si : [0,1] \to U \cap X^\o$ with $\si(0)=y$ and $\si(1)=z$.
   We find a finite partition $0 = t_0 < t_1 < \cdots < t_{h-1} < t_h = 1$ of $[0,1]$
   such that the points $z_i := \si(t_i)$ have the following properties:
   If $\ep>0$ is such that the balls $B_i = B(z_i,\ep)$ are contained in $U \cap X^\o$,
   then for each $0 \le i \le h-1$ there is a piece $A_{j_i}$ which intersects
   $B_i$ and $B_{i+1}$, and all pieces in the list $A_{j_0},A_{j_1},\ldots$
   are different.
   Using that the sets $A_j$ are $M$-quasiconvex, it is now easy to find a
   rectifiable path $\ga$ in $X^\o$ from $y$ to $z$ of length at most
    $C \on{diam}(U)$, where $C$ depends only on $M$ and the number of pieces $A_j$.
\end{proof}

\subsection{Uniformly polynomially cuspidal sets} \label{sec:UPC}

Recall the definition of a closed UPC set $X \subseteq \R^d$ from the introduction:
there exist positive constants $M,m$ and a positive integer $N$ such that
for each $x \in X$ there is a polynomial curve $h_x : \R \to \R^d$ of degree at most $N$
satisfying
\begin{enumerate}
  \item $h_x(0)=x$,
  \item $\on{dist}(h_x(t),\R^d \setminus X) \ge M t^m$ for all $x \in X$ and $t \in [0,1]$.
\end{enumerate}
In particular, $h_x((0,1]) \subseteq X^\o$. 
Note that a UPC set is necessarily fat. All H\"older sets $X \in \sH(\R^d)$ are UPC, which is an easy consequence of the definition.
O-minimal structures that are not polynomially bounded contain sets (e.g.\ infinitely flat cusps) that are not UPC.
Also polynomially bounded o-minimal structures may contain fat sets that are not UPC.
For instance, $X=\{(x,y) \in \R^2 : 0\le x^{\sqrt 2} \le y \le x^{\sqrt 2} + x^2\}$ is definable in $\R_{\text{an}}^\R$,
but every $\cC^\infty$-curve in $X$ through the origin must vanish to infinite order (see \Cref{ex:irrational}).
It was shown in \cite{Pierzchal-a:2005uk} that a compact subset $X$ of $\R^2$ definable in some polynomially
bounded o-minimal structure is UPC if and only if it is fat and
for each $x \in X$, each $r>0$, and each connected component $S$ of $X^\o \cap B(x,r)$ with $x \in \ol S$
there is a polynomial curve $c : (0,1) \to S$ such that $c(t) \to x$ as $t \to 0$.

On the other hand there are lots of examples of UPC sets in $\R^d$ (besides H\"older sets):
\begin{itemize}
  \item Fat compact subanalytic sets; cf.\ \cite[Corollary 6.6]{PawluckiPlesniak86}.
  \item Fat compact sets definable in $\R_\cQ$, the o-minimal expansion of the real field by restricted functions in a suitable
  quasianalytic class $\cQ$; cf.\ \cite{Pierzchal-a:2005uk}.
  \item Fat compact sets definable in a certain substructure of the structure generated by generalized power series; cf.\
  \cite{Pierzchal-a:2012uv}.
\end{itemize}

\subsection{Smooth rectilinearization} \label{ass:smoothrect}

Let $X \subseteq \R^d$ be a fat compact definable set.
We say that \emph{$X$ admits smooth rectilinearization} if
there is a finite number of  definable $\cC^\infty$-maps $\ps_j : \R^d \to \R^d$
  such that
  \[
    \ps_j((-1,1)^d) \subseteq X^\o, \quad \text{ for each } j, \quad \text { and } \quad
    \bigcup_j \ps_j([-1,1]^d) = X.
  \]
For instance, if $X$ is subanalytic or definable in $\R_{\cQ}$, where $\cQ$ is a suitable quasianalytic class, then it admits smooth rectilinearization;
cf.\ \cite{Hironaka73}, \cite{PawluckiPlesniak86},  \cite{BM97,BM04}, and \cite{RolinSpeisseggerWilkie03}.

\begin{lemma} \label{lem:smoothrect}
  Let $X \subseteq \R^d$ be a fat compact definable set admitting smooth rectilinearization.
  There is a finite number of definable $\cC^\infty$-maps $\vh_j : \R^d \times \R \to \R^d$ such that
  \begin{gather*}
    \vh_j(I^d \times (0,1]) \subseteq X^\o, \quad \text{ for each }j,  \quad \text { and } \quad
    \bigcup_j \vh_j(I^d \times \{0\}) = X,
  \end{gather*}
  where $I^d := [-1,1]^d$.
\end{lemma}

\begin{proof}
  Compose $\ps_j$ with $(x_1,\ldots,x_d,t) \mapsto (x_1(1-t),\ldots,x_d(1-t))$.
\end{proof}

The following arguments are taken from the proof of \cite[Theorem 6.4]{PawluckiPlesniak86} and adapted to the definable setting.
For each $j$, the function
\[
	I^d \times [0,1] \ni (y,t) \mapsto \on{dist}(\vh_j(y,t), \R^d \setminus X)
\]
is definable.
By the {\L}ojasiewicz inequality \eqref{eq:Lojasiewicz},
there exist $L>0$ and $m \in \N_{\ge 1}$ such that
\begin{equation} \label{eq:Lojasiewicz2}
    	\on{dist}(\vh_j(y,t), \R^d \setminus X) \ge L t^{m}, \quad  (y,t) \in I^d\times [0,1].
\end{equation}
The constants $L$, $m$ may be assumed to be independent of $j$ by taking the minimum and maximum, respectively.
Write
\[
	\vh_j(y,t)  = T_j(y,t) + t^{m+1} Q_j(y,t), \quad (y,t) \in \R^d \times \R,
\]
where $T_j(y,\cdot)$ is the Taylor polynomial at $0$ of degree $m$ of $\vh_j(y,\cdot)$.
If we choose $\de \in (0,1]$ such that
$|tQ_j(y,t)| \le L/2$ for all $j$, $y \in I^d$, and $t \in [0,\de]$, then
\[
	\on{dist}(T_j(y,t), \R^d \setminus X) \ge L t^m - \frac{L}{2} t^m = \frac{L}{2} t^m, \quad (y,t) \in I^d \times [0,\de].
\]
Replacing $t$ by $\de t$, we obtain
\[
	\on{dist}(T_j(y,\de t), \R^d \setminus X) \ge  M t^m, \quad (y,t) \in I^d \times [0,1],
\]
where $M:= \frac{1}{2}L \de^m$. Clearly, $\bigcup_j T_j(I^d \times \{0\}) = \bigcup_j \vh_j(I^d \times \{0\}) = X$.
From this it is easy to conclude

\begin{proposition} \label{prop:UPC}
  Any fat compact definable set $X \subseteq \R^d$  admitting smooth rectilinearization
  is UPC.
\end{proposition}

We recall that the reciprocal $\al = \frac{1}{m}$ of the integer $m$ that appears in \eqref{eq:Lojasiewicz2}
is called a UPC-index of $X$.

The following lemma will be an important tool in the proof of \Cref{main:B}.

\begin{lemma} \label{lem:bdder}
  Let $X \subseteq \R^d$ be a fat compact definable set admitting smooth rectilinearization and
  $\al$ a UPC-index of $X$.
  Let $n \in \N_{\ge 1}$ and $\om$ a modulus of continuity.
  For each $f \in \cA^{np(\al),\om}(X)$
  the Fr\'echet derivatives $f^{(k)}$, $k \le n$, are bounded on $X^\o$.
\end{lemma}

Note that the derivatives of $f \in \cA^{np(\al),\om}(X)$ exist up to order $np(\al)$ 
(by \cite[Theorem 2]{Boman67}, \cite[Th\'eor\`eme 1]{Faure89}), but
$f^{(n+1)}$ need not be globally bounded on $X^\o$ as seen in \Cref{rem:unbounded}.

\begin{proof}
  Let $k \le n$ be fixed.
  For contradiction, suppose that there is a sequence $(x_\ell)$ in $X^\o$ such that
  $\{f^{(k)}(x_\ell): \ell \in \N\}$ is unbounded.
  Let $\vh_j$ be the maps from \Cref{lem:smoothrect}.
  After passing to a subsequence, we may assume that $x_\ell \in \vh_{j_0}(I^d \times \{0\})$ for all $\ell$
  and some $j_0$.
  Choose $y_\ell \in I^d$ such that $\vh_{j_0}(y_\ell,0) = x_\ell$. Since $I^d$ is compact,
  after repeatedly passing to subsequences we may assume that $y_\ell$ converges to $y \in I^d$
  and that $y_\ell - y$ tends fast to $0$.
  The infinite polygon through the points $y_\ell$ and $y$ can be parameterized by a $\cC^\infty$-curve $c : \R \to I^d$
  such that $c(\frac{1}\ell) = y_\ell$ and $c(0) = y$ (cf.\ \cite[Lemma 2.8]{KM97}).
  Then $s \mapsto \vh_{j_0}(c(s),0)$ is a $\cC^\infty$-curve in $X$ through the points $x_\ell$ and $x = \vh_{j_0}(y,0)$.

  Let $v \in \mathbb S^{d-1}$ be arbitrary.
  By \Cref{main:A},
  \begin{equation*}
    (s,t_1,t_2) \to f\big(\vh_{j_0}(c(s),t_1) + t_2 v\big)
  \end{equation*}
  is of class $\cC^n$ for small $s \in \R$, $t_1 \ge 0$, and $|t_2| \le \frac{L}2 t_1^m$.
  Indeed, such $(s,t_1,t_2)$ range over an $\al$-set (where $\al = \frac{1}{m}$) and the point $\vh_{j_0}(c(s),t_1) + t_2$ lies in $X$, since,
  by \eqref{eq:Lojasiewicz2},
  \begin{align*}
    \on{dist}(\vh_{j_0}(c(s),t_1) + t_2 v, \R^d \setminus X) &\ge \on{dist}(\vh_{j_0}(c(s),t_1), \R^d \setminus X) - |t_2|
    \\
    &\ge L t_1^m - \frac{L}2 t_1^m = \frac{L}2 t_1^m.
  \end{align*}
  But this implies that the directional derivative $d_v^k f(x_\ell)$ is bounded in $\ell$.
  Since $v$ was arbitrary, $f^{(k)}(x_\ell)$
  is bounded (e.g.\ in view of the polarization formula \cite[(7.13.1)]{KM97}),
  a contradiction.
\end{proof}

As a by-product we obtain

\begin{proposition}
  Let $X \subseteq \R^d$ be a fat compact definable set admitting smooth rectilinearization.
  The $c^\infty$-topology of $X$ coincides with the trace topology of $\R^d$.
\end{proposition}

\begin{proof}
  Let $A$ be a $c^\infty$-closed subset of $X$ and let $\ol A$ be the closure of $A$ in $\R^d$.
  We have to show that $\ol A \subseteq A$. Let $x \in \ol A$ and $x_\ell$ a sequence in $A$ with $x_\ell \to x$.
  If we find a $\cC^\infty$-curve in $X$ through a subsequence of $x_\ell$ and through $x$, we have $x \in A$ and are done.
  Since $X$ admits smooth rectilinearization, there exists a $\cC^\infty$-map $\ps : \R^d \to \R^d$
  and an infinite subsequence of $x_\ell$, again denoted by $x_\ell$, which is contained in
  $\ps([-1,1]^d)$.
  Choose $y_{\ell} \in [-1,1]^d$ such that $\ps(y_{\ell}) = x_{\ell}$.
  As in the proof of \Cref{lem:bdder} we may pass to a fast converging subsequence $y_\ell \to y \in [-1,1]^d$
  and find a $\cC^\infty$-curve $c$ in $[-1,1]^d$ which passes through this subsequence and $y$.
  Then $\ps \o c$ is a $\cC^\infty$-curve in $X$ through the corresponding $x_\ell = \ps(y_{\ell})$ and $x= \ps(y)$.
\end{proof}

\subsection{Proof of \Cref{main:B}}

  Let $X \subseteq \R^d$ be a simple fat compact definable set
  admitting smooth rectilinearization
  and let $\al$ be a UPC-index for $X$.
  Let $n$ be a positive integer and $\om$ a modulus of continuity.
  Let $f \in \cA^{np(\al),\om}(X)$.
  Then the Fr\'echet derivatives $f^{(k)}$ exist and are continuous on $X^\o$ for all $k \le np(\al)$, 
  by \cite[Theorem 2]{Boman67}, \cite[Th\'eor\`eme 1]{Faure89},
  and they are globally bounded on $X^\o$ for all $k \le n$, by \Cref{lem:bdder}.
  It remains to show that
  for all $k \le n-1$ they extend continuously to $\p X$.

  Fix $x \in \p X$.
  By \Cref{prop:UPC}, there exist $m,N \in \N_{\ge 1}$ with $\al = \frac{1}{m}$, $M >0$, and a polynomial curve $h_x : \R \to \R^d$
  of degree at most $N$ such that
  \begin{enumerate}
      \item $h_x(0)=x$,
      \item $\on{dist}(h_x(t),\R^d \setminus X) \ge Mt^m$ for all $t \in (0,1]$.
  \end{enumerate}
  Then $h_x(t) - x$ vanishes to finite order.
  So there is a positive integer $j=j(x)$ such that $h_x(t) - x = t^j \tilde h_x(t)$, where $\tilde h_x(0) \ne 0$.
  Set $v_1 := \frac{\tilde h_x(0)}{|\tilde h_x(0)|} \in \mathbb S^{d-1}$.
  Choose $d-1$ directions $v_2,\ldots, v_d \in \mathbb S^{d-1}$ such that $v_1,v_2,\ldots, v_d$ are linearly independent and
  consider the map $\Ps_{x,v} : \R^d \to \R^d$ defined by
  \begin{equation} \label{eq:Ps}
    \Ps_{x,v}(t_1,t_2,\ldots,t_d) := h_x(t_1) + t_2 v_2 + \cdots + t_d v_d.
  \end{equation}
  The restriction of $\Ps_{x,v}$ to
  \begin{equation} \label{eq:Y}
    Y :=\big\{(t_1,\ldots,t_d) \in \R^d : t_1 \in (0,\de),\, |t_j| < \tfrac{M}{2(d-1)} t_1^m \text{ for } 2\le j \le d\big\},
  \end{equation}
  for small $\de>0$,
  is a diffeomorphism onto the open subset $H_{x,v} := \Ps_{x,v}(Y)$
  of $X^\o$ and it extends to a homeomorphism between $Y\cup\{0\}$ and $H_{x,v} \cup \{x\}$.
  Indeed,
  \begin{align*}
    \on{dist}(\Ps_{x,v}(t),\R^d \setminus X) &\ge \on{dist}(h_x(t_1),\R^d \setminus X) - |t_2| - \cdots - |t_d|
    \\
    & > M t_1^m - \frac{M}{2} t_1^m = \frac{M}{2} t_1^m >0,
  \end{align*}
  for $t \in Y$.
  This also shows that $\Ps_{x,v}(\ol Y)\subseteq X$.
  Since $f$ is of class $\cC^{n}$ in $X^\o$, we have
  \begin{equation} \label{eq:defcontder}
    \p_{t_2}^{j_2} \cdots \p_{t_d}^{j_d} (f \o \Ps_{x,v})(t) =  d_{v_2}^{j_2} \cdots d_{v_d}^{j_d} f(\Ps_{x,v}(t)),
   \end{equation}
  for all $t \in Y$ and $0 \le j_2 + \cdots + j_d \le n$; actually even up to order $n p(\al)$ but we will not need this.
  The function $f \o \Ps_{x,v}|_{\ol Y}$ belongs to $\cA^{np(\al),\om}(\ol Y)$ and hence to $\cC^{n}(\ol Y)$,
  by \Cref{main:A}, as $\ol Y$ is an $\al$-set.
  It follows that
  the left-hand side of \eqref{eq:defcontder} extends continuously to $t=0$
  for all $0 \le j_2 + \cdots + j_d \le n$.
  So the directional derivatives
  $d_{v_2}^{j_2} \cdots d_{v_d}^{j_d} f$, $0 \le j_2 + \cdots + j_d \le n$, extend continuously from $H_{x,v}$ to $x$.

  If we perturb the directions $v_2,\ldots, v_d$ a little such that $v_1,v_2,\ldots,v_d$ remain
  linearly independent and take the intersection $H_x$ of the corresponding sets $H_{x,v}$,
  then $H_x$ still is an open subset of $X^\o$ with $h_x(t) \in H_x$ for small $t>0$ and $x \in \ol H_x$.
  So the directional derivatives $d_{w_2}^{j_2} \cdots d_{w_d}^{j_d} f$, $0 \le j_2 + \cdots + j_d \le n$, extend continuously from $H_{x}$ to $x$
  for all $w_2,\ldots,w_d$ near $v_2,\ldots, v_d$.
  Thus the Fr\'echet derivatives $f^{(k)}$, for $k \le n$, extend continuously from
  $H_x$ to $x$ (in view of the polarization formula \cite[(7.13.1)]{KM97}).

  For each $x \in \p X$
  we define
  \[
    f^{(k)}(x) := \lim_{H_x \ni y \to x} f^{(k)}(y), \quad k \le n.
  \]
  It remains to prove that the so defined extension of $f^{(k)}$ to $\p X$ is continuous at $\p X$ if $k \le n-1$.
  To this end fix $x \in \p X$ and let
  $(x_j)$ and $(y_j)$ be two sequences in
  $X^\o$ converging to $x$.
  By \Cref{lem:quasiconvex}, for each $\ep>0$ there exists $j_0 \in \N$ such that for all $j \ge j_0$
  the points $x_j$ and $y_j$ can be joined by a rectifiable path $\ga_j$ in $X^\o$ of length $\ell(\ga_j) \le \ep$.
  Thus
  \[
  \|f^{(k)}(x_j) - f^{(k)}(y_j)\|_{L^k(\R^d,\R)} \le \left(\sup_{z \in \ga_j} \|f^{(k+1)}(z)\|_{L^{k+1}(\R^d,\R)}\right) \, \ell(\ga_j)
  \]
  tends to $0$ as $j \to \infty$ if $k \le n-1$, since $f^{(k+1)}$ is globally bounded on $X^\o$, by \Cref{lem:bdder}.
  If we assume that the sequence $(x_j)$ lies in $H_x$,
  we obtain
  \[
    f^{(k)}(x)
    = \lim_{X^\o \ni y \to x} f^{(k)}(y), \quad k \le n-1.
  \]
  Finally, suppose that $\p X \ni x_j \to x$.
  Choose $y_j \in H_{x_j} \cap B(x_j,j^{-1})$.
  Then
  \begin{multline*}
    \|f^{(k)}(x) - f^{(k)}(x_j)\|_{L^k(\R^d,\R)}
    \\
    \le \|f^{(k)}(x) - f^{(k)}(y_j)\|_{L^k(\R^d,\R)} +
     \|f^{(k)}(x_j) - f^{(k)}(y_j)\|_{L^k(\R^d,\R)}
  \end{multline*}
  tends to $0$ as $j \to \infty$.
  Thus $f^{(k)}$ extends continuously to $x$.

  Now suppose that $Z$ is any non-empty subset of $X$ and $f \in \cA^{np(\al),\om}_Z(X)$. Fix $x \in Z$.
  We want to show that $f^{(k)}(x) = 0$ for all $k \le n-1$; actually, it is true even for $k = n$.
  The assertion is easy to see if $x \in X^\o$. So let us assume that $x \in \p X$ and let $\Ps_{x,v}$ be the map from \eqref{eq:Ps} and $Y$ the set from \eqref{eq:Y}.
  Then $f \o \Ps_{x,v}|_{\ol Y}$ belongs to $\cA^{np(\al),\om}_{\{0\}}(\ol Y)$ and thus to $\cC^n_{\{0\}}(\ol Y)$, by \Cref{main:A}.
  In view of \eqref{eq:defcontder} and the perturbation argument shortly after \eqref{eq:defcontder},
  we may conclude that
  $f^{(k)}(x) =0$ for $k \le n$.
  \Cref{main:B} is proved.

\subsection{Proof of \Cref{main:C}}
  This follows in analogy to the proof of \Cref{main:Aa} (see \Cref{sec:proofmain:Aa}) from \Cref{main:B},
  where we use Whitney regularity \ref{sec:Wregular} instead of \Cref{prop:alpharegular}.

\subsection{Weak flatness on UPC sets}

Let $X\subseteq \R^d$ be a fat closed set and $x \in X$. Let $r\ge 0$.
A function $f: X \to \R$ is called \emph{weakly $r$-flat at $x$ in $X$} if
\[
  \frac{|f(y)|}{|y-x|^r} \to 0 \quad \text{ as } X \ni y \to x.
\]
It is called \emph{weakly $\infty$-flat at $x$ in $X$} if it is
weakly $n$-flat at $x$ for each $n \in \N$.

\begin{corollary} \label{cor:weakflat}
  Let $X \subseteq \R^d$ be a closed UPC set (not necessarily simple or definable) and let $\al$ be a
  UPC-index of $X$.  Let $x \in X$. Let $n \in \N$ and $\om$ a modulus of continuity.
  Assume that $f \in \cA^{n p(\al),\om}(X) \cap \cC^{n}(X)$ is weakly $\frac{np(\al)}{2}$-flat at $x$ in $X$.
  Then $f$ is $n$-flat at $x$.
\end{corollary}

\begin{proof}
   Since $X$ is UPC  and $\al$ is a UPC-index of $X$, we find as in the proof of \Cref{main:B}
   an $\al$-set $\ol Y$ and a map $\Ps_{x,v} : \R^d \to \R^d$ such that
   $\Ps_{x,v}(\ol Y) \subseteq X$, $\Ps_{x,v}(0) = x$, and
   $g:= f \o \Ps_{x,v}|_{\ol Y}$ belongs to $\cA^{n p(\al),\om}(\ol Y)$.
   In view of \Cref{main:A}, \eqref{eq:defcontder},  and the perturbation argument shortly after \eqref{eq:defcontder},
   it suffices to show that $g$ actually belongs to $\cA^{n p(\al),\om}_{\{0\}}(\ol Y)$.
   So we fix $c \in \cC^\infty(\R,\ol Y)$ with $c(0)=0$ and check that $(g \o c)^{(j)}(0)=0$ for all $j \le np(\al)$.
   The shape of $\ol Y$ imposes that $c$ vanishes to order at least $2$ at $0$, i.e., $c(t) = t^2 \tilde c(t)$.
   Since $f$ is weakly $r:=\frac{np(\al)}{2}$-flat at $x$ in $X$,
   also $g$ is weakly $r$-flat at $0$ in $\ol Y$:
   \[
      \frac{|g(y)|}{|y|^r} = \frac{|f(\Ps_{x,v}(y))|}{|\Ps_{x,v}(y) -x|^r} \frac{|\Ps_{x,v}(y) -x|^r}{|y|^r}\to 0 \quad \text{ as } \ol Y \ni y \to 0,
   \]
   since $\Ps_{x,v}$ is locally Lipschitz. Consequently,
   \[
     \frac{|g(c(t))|}{|c(t)|^{r}} = \frac{|g(c(t))|}{|t|^{2r}|\tilde c(t)|^{r}}  \to 0 \quad \text{ as } t \to 0,
   \]
   and we see that $g \o c$ is weakly $np(\al)$-flat at $0$ in $\R$.
   It follows that $(g \o c)^{(j)}(0)=0$ for all $j \le np(\al)$ and we are done.
\end{proof}

A similar result has been obtained in \cite[Satz 1.4]{Spallek:1977wp}:
A $\cC^n$-function $f$ defined on a neighborhood of $x$ in $\R^d$ is $n$-flat at $x$,
if $f$ is weakly $n p$-flat at $x$ in $X$ and $X\subseteq \R^d$ contains a sequence of balls $B(x_k,r_k)$
such that $|x_k - x|^p/r_k \to 0$.
See also \cite{Izumi:2002wq}.
In the language of \cite{Izumi:2002wq}, the corollary implies that the $\cC^\infty$-Spallek function $S^\infty_{X,x}$
of a closed UPC set $X$ with UPC-index $\al$ and $x \in X$ satisfies $S^\infty_{X,x}(n) \le n p(\al)$.

\subsection{Open problem}

Our proof of \Cref{main:B} uses that the set $X$ admits smooth rectilinearization and (as a consequence)
is UPC. We do now know if these assumptions can be relaxed.
There are polynomially bounded o-minimal expansions of the real field in which sets that lack these
properties are definable. For instance, the set $X$ in \Cref{ex:irrational} below is definable in the structure
$\R_{\text{an}}^\R$, but it is not UPC (cf.\ \Cref{sec:UPC}).
Any $\cC^\infty$-curve $c$ in $X$ through the boundary point $0$ has to vanish to infinite order on $c^{-1}(0)$.
It could be an indication that $\cC^\infty$-curves may no be enough to detect derivatives at some boundary points
even though the set has finite cuspidality.
Be that as it may,
the $\cC^\infty$-curves certainly do not discriminate points of flatness:
The function $f(x,y) = x$ belongs to $\cA^\infty_{\{0\}}(X)$ but $\p_x f(0) =1$.
Furthermore, the smooth function $g(x,y) =e^{-1/x^2}$ if $x \ne 0$ and $g(0,y)=0$
is not analytic near the origin,
but, trivially, $g \o c$ is analytic for all analytic curves $c$ in $X$; so the Bochnak--Siciak theorem fails on $X$,
while it holds on simple fat closed subanalytic sets and H\"older sets; cf.\
\cite[Theorem 1.16 and Corollary 1.17]{Rainer18}.

\begin{example} \label{ex:irrational}
  Let $\si>1$ be an irrational number and $p> \si$ an integer.
  Consider
  \[
    X := \{(x,y) \in \R^2 : 0 \le x \le 1, \, x^\si \le y \le x^{\si} + x^p\}.
  \]
  Each $\cC^\infty$-curve $c : \R \to X$ must be infinitely flat on $c^{-1}(0)$.
  Suppose for contradiction that
  $c(0) = 0$ and
  $c(t) = (x(t),y(t))$ vanishes only to finite order at $t=0$.
  There exist positive integers $k,\ell$ such that
  $x(t) = t^k \tilde x(t)$ and $y(t) = t^\ell \tilde y(t)$, where $\tilde x, \tilde y$ are smooth
  and either $\tilde x(0)\ne 0$ or $\tilde y(0) \ne 0$. 
  Then
  \[
    0 \le t^{\si k} \tilde x(t)^\si \le t^\ell \tilde y(t) \le t^{\si k} \tilde x(t)^\si + t^{p k} \tilde x(t)^p, \quad t \ge 0.
  \]
  Several cases must be discussed:
  \begin{enumerate}
    \item $\tilde x(0) \ne 0$ and $\si k -\ell <0$ is impossible, since it would mean that $\tilde y(t)$ is unbounded at $0$.
    \item $\tilde x(0) \ne 0$ and $\si k -\ell >0$ implies that $\tilde y(0)=0$. So we may replace $\ell$ by $\ell +1$ and repeat the reasoning.
    We either end up in case (1) or $y(t)$ vanishes to infinite order at $t=0$. But the latter contradicts $\tilde x(0) \ne 0$. So case (2) is impossible.
    \item $\tilde y(0) \ne 0$ and $\si k -\ell >0$ is impossible.
    \item $\tilde y(0) \ne 0$ and $\si k -\ell <0$ implies that $\tilde x(0)=0$. As in case (2), we may replace $k$ by $k+1$ and repeat the argument.
    This leads to case (3) or $x(t)$ vanishes to infinite order. But then also $y(t)$ vanishes to infinite order at $t=0$, a contradiction.
  \end{enumerate}
  It seems to be unknown if $X$ satisfies the Markov inequality \eqref{MI} (cf.\ \cite[p. 649]{Pierzchal-a:2012uv}), but
  it has the Whitney extension property WEP (e.g., by \cite[Theorem 3.15]{Frerick:2007aa}) and hence admits a weaker
  inequality of Markov type, by \cite[Theorem 4.6]{Frerick:2007aa}: 
  for all $\th \in (0,1)$ there exist $C,r\ge 1$ such that 
  \begin{equation} \label{eq:wMI}
    \sup_{x \in X} |\nabla p(x)| \le C (\deg p)^r  \sup_{x \in X} |p(x)|^\th  \sup_{x \in K} |p(x)|^{1-\th} 
  \end{equation}
  for all real polynomials $p$,
  where $K \subseteq \R^2$ is any compact set with $X \subseteq K^\o$. 
\end{example}

\subsection*{Acknowledgement}
I am grateful to the anonymous referee for valuable remarks that led to an improvement of the presentation.

\def\cprime{$'$}
\providecommand{\bysame}{\leavevmode\hbox to3em{\hrulefill}\thinspace}
\providecommand{\MR}{\relax\ifhmode\unskip\space\fi MR }
\providecommand{\MRhref}[2]{%
  \href{http://www.ams.org/mathscinet-getitem?mr=#1}{#2}
}
\providecommand{\href}[2]{#2}


\begin{thebibliography}{10}

\bibitem{Bierstone80a}
E.~Bierstone, \emph{Differentiable functions}, Bol. Soc. Brasil. Mat.
  \textbf{11} (1980), no.~2, 139--189.

\bibitem{BM97}
E.~Bierstone and P.~D. Milman, \emph{Canonical desingularization in
  characteristic zero by blowing up the maximum strata of a local invariant},
  Invent. Math. \textbf{128} (1997), no.~2, 207--302.

\bibitem{BM04}
\bysame, \emph{Resolution of singularities in {D}enjoy-{C}arleman classes},
  Selecta Math. (N.S.) \textbf{10} (2004), no.~1, 1--28.

\bibitem{Boman67}
J.~Boman, \emph{Differentiability of a function and of its compositions with
  functions of one variable}, Math. Scand. \textbf{20} (1967), 249--268.

\bibitem{Bos:1995wj}
L.~P. {Bos} and P.~D. {Milman}, \emph{{Sobolev-Gagliardo-Nirenberg and Markov
  type inequalities on subanalytic domains}}, {Geom. Funct. Anal.} \textbf{5}
  (1995), no.~6, 853--923.

\bibitem{LlaveObaya99}
R.~de~la Llave and R.~Obaya, \emph{Regularity of the composition operator in
  spaces of {H}{\"o}lder functions}, Discrete Contin. Dynam. Systems \textbf{5}
  (1999), no.~1, 157--184.

\bibitem{DelfourZolesio11}
M.~C. {Delfour} and J.-P. {Zol\'esio}, \emph{{Shapes and geometries. Metrics,
  analysis, differential calculus, and optimization. 2nd ed.}}, 2nd ed. ed.,
  Philadelphia, PA: Society for Industrial and Applied Mathematics (SIAM), 2011.

\bibitem{Faure89}
C.-A. Faure, \emph{Sur un th\'eor\`eme de {B}oman}, C. R. Acad. Sci. Paris
  S\'er. I Math. \textbf{309} (1989), no.~20, 1003--1006.

\bibitem{Frerick:2007aa}
L.~Frerick, \emph{Extension operators for spaces of infinite differentiable
  {W}hitney jets}, J. Reine Angew. Math. (Crelle's
  Journal) \textbf{602} (2007), 123--154.

\bibitem{Glaeser63R}
G.~Glaeser, \emph{Racine carr\'ee d'une fonction diff\'erentiable}, Ann. Inst.
  Fourier (Grenoble) \textbf{13} (1963), no.~2, 203--210.

\bibitem{Grisvard85}
P.~Grisvard, \emph{Elliptic problems in nonsmooth domains}, Monographs and
  Studies in Mathematics, vol.~24, Pitman (Advanced Publishing Program),
  Boston, MA, 1985.

\bibitem{Hironaka73}
H.~Hironaka, \emph{Introduction to real-analytic sets and real-analytic maps},
  Quaderni dei Gruppi di Ricerca Matematica del Consiglio Nazionale delle
  Ricerche. Istituto Matematico "L. Tonelli'' dell'Universit{\`a} di Pisa,
  Pisa, 1973.

\bibitem{Izumi:2002wq}
S.~Izumi, \emph{Flatness of differentiable functions along a subset of a real
  analytic set}, Journal d{\textquotesingle}Analyse Math{\'{e}}matique
  \textbf{86} (2002), no.~1, 235--246.

\bibitem{JorisPreissmann90}
H.~Joris and E.~Preissmann, \emph{Quotients of smooth functions}, Kodai Math.
  J. \textbf{13} (1990), no.~2, 241--264.

\bibitem{Kriegl97}
A.~Kriegl, \emph{Remarks on germs in infinite dimensions}, Acta Math. Univ.
  Comenian. (N.S.) \textbf{66} (1997), no.~1, 117--134.

\bibitem{KM97}
A.~Kriegl and P.~W. Michor, \emph{The convenient setting of global analysis},
  Mathematical Surveys and Monographs, vol.~53, American Mathematical Society,
  Providence, RI, 1997, http://www.ams.org/online\_bks/surv53/.

\bibitem{KurdykaParusinski06}
K.~Kurdyka and A.~Parusi{\'n}ski, \emph{Quasi-convex decomposition in o-minimal
  structures. {A}pplication to the gradient conjecture}, Singularity theory and
  its applications, Adv. Stud. Pure Math., vol.~43, Math. Soc. Japan, Tokyo,
  2006, pp.~137--177.

\bibitem{Michor:2020ty}
P.~W. {Michor}, \emph{{Manifolds of mappings for continuum mechanics}},
  Geometric continuum mechanics, Cham: Birkh\"auser, 2020, pp.~3--75.

\bibitem{Miller:1994ue}
C.~Miller, \emph{Exponentiation is hard to avoid}, Proc. Amer. Math. Soc.
\textbf{122} (1994), no.~1, 257--259.

\bibitem{NenningRainer16}
D.~N. Nenning and A.~Rainer, \emph{On groups of {H}\"older diffeomorphisms and
  their regularity}, {T}rans. {A}mer. {M}ath. {S}oc. \textbf{370} (2018),
  no.~8, 5761--5794.

\bibitem{PawluckiPlesniak86}
W.~Paw{\l}ucki and W.~Ple{\'s}niak, \emph{Markov's inequality and {$C^\infty$}
  functions on sets with polynomial cusps}, Math. Ann. \textbf{275} (1986),
  no.~3, 467--480.

\bibitem{PawluckiPlesniak88}
\bysame, \emph{Extension of {$C^\infty$} functions from sets with polynomial
  cusps}, Studia Math. \textbf{88} (1988), no.~3, 279--287.

\bibitem{Pierzchal-a:2005uk}
R.~Pierzcha{\l}a, \emph{U{PC} condition in polynomially bounded o-minimal
  structures}, J. Approx. Theory \textbf{132} (2005), no.~1, 25--33.

\bibitem{Pierzchal-a:2012uv}
\bysame, \emph{Markov's inequality in the o-minimal structure of convergent
  generalized power series}, Adv. Geom. \textbf{12} (2012), no.~4, 647--664.

\bibitem{Plesniak:1990aa}
W.~Ple{\'{s}}niak, \emph{{Markov{\textquotesingle}s inequality and the
  existence of an extension operator for $C^\infty$ functions}}, Journal of
  Approximation Theory \textbf{61} (1990), no.~1, 106--117.

\bibitem{Rainer18}
A.~Rainer, \emph{Arc-smooth functions on closed sets}, Compos. Math.
  \textbf{155} (2019), 645--680.

\bibitem{Roberts:wy}
D.~M. Roberts and A.~Schmeding, \emph{Extending {W}hitney's extension theorem:
  nonlinear function spaces}, Ann. Inst. Fourier (Grenoble) \textbf{71} (2021),
  no.~3, 1241--1286.

\bibitem{RolinSpeisseggerWilkie03}
J.-P. Rolin, P.~Speissegger, and A.~J. Wilkie, \emph{Quasianalytic
  {D}enjoy-{C}arleman classes and o-minimality}, J. Amer. Math. Soc.
  \textbf{16} (2003), no.~4, 751--777.

\bibitem{Spallek:1977wp}
K.~Spallek, \emph{{l-Platte Funktionen auf semianalytischen Mengen}},
  Math. Ann. \textbf{227} (1977), no.~3, 277--286.

\bibitem{vandenDries98}
L.~van~den Dries, \emph{Tame topology and o-minimal structures}, London
  Mathematical Society Lecture Note Series, vol. 248, Cambridge University
  Press, Cambridge, 1998.

\bibitem{vandenDriesMiller96}
L.~van~den Dries and C.~Miller, \emph{Geometric categories and o-minimal
  structures}, Duke Math. J. \textbf{84} (1996), no.~2, 497--540.

\bibitem{Whitney34a}
H.~Whitney, \emph{Analytic extensions of differentiable functions defined in
  closed sets}, Trans. Amer. Math. Soc. \textbf{36} (1934), no.~1, 63--89.

\end{thebibliography}
\end{document}